\documentclass[a4paper,10pt]{article}
\usepackage{geometry}
\usepackage[parfill]{parskip}

\usepackage{amsmath}
\usepackage{amsfonts}
\usepackage{amssymb}
\usepackage{amsthm}
\usepackage{color}
\usepackage{graphicx}
\usepackage{relsize}

\theoremstyle{plain}
\newtheorem{definition}{Definition}

\newtheorem{lemma}{Lemma}

\newtheorem{theorem}{Theorem}
\newtheorem{corollary}{Corollary}

\theoremstyle{remark}

\begin{document}

\title{Finite Element Approximation of the Modified Maxwell's Stekloff Eigenvalues}
\author{Bo Gong$^1$, Jiguang Sun$^2$, and Xinming Wu$^{3}$\\
\small{$^1$ Beijing Computational Science Research Center, Beijing 100193, China.}\\
\small{$^2$ Department of Mathematical Sciences, Michigan Technological University, Houghton, MI 49931, USA.}\\
\small{$^3$ Shanghai Key Laboratory for Contemporary Applied Mathematics,}\\
\small{School of Mathematical Sciences, Fudan University, Shanghai 200433, China.}}
\date{}
\maketitle

\begin{abstract}
The modified Maxwell's Stekloff eigenvalue problem arises recently from the inverse electromagnetic scattering theory for inhomogeneous media. This paper contains a rigorous analysis of both the eigenvalue problem and the associated source problem on Lipschitz polyhedra.  A new finite element method is proposed to compute Stekloff eigenvalues. By applying the Babu\v{s}ka-Osborn theory, we prove an error estimate without additional regularity assumptions. Numerical results are presented for validation.
\end{abstract}

\section{Introduction}
Target signature using transmission eigenvalues \cite{Harris2014} or Stekloff eigenvalues \cite{Cakoni2016} has attracted a lot of attention in the context of non-destructive testing. 
These eigenvalues can be obtained through the scattered field and used to reconstruct the properties of the scatterer. 
Recently, \cite{Camano2017} extended the concept of Stekloff eigenvalues to the Maxwell's equations and obtained a new eigenvalue problem. 
They show that the so-called modified Maxwell's Stekloff eigenvalues can be used to detect changes in the scatterer using remote measurements. 
In this paper, we shall focus on the numerical computation of the Stekloff eigenvalues. 
Our purpose is to analyze this eigenvalue problem and propose a convergent finite element method.

There are two relevant spaces to be approximated: the curl space of the domain and the $H^1$ space of the boundary. 
We use the curl conforming edge elements \cite{Nedelec1980} for the former and the Lagrange elements for the latter.
Among the techniques used to analyze the edge elements, the discrete compactness property is a powerful one, 
which was discussed in \cite{Kikuchi1989} for the lowest-order edge element. 
It was further analyzed in \cite{Monk2003} for the Maxwell's equations with impedance boundary conditions. 
The analysis holds when the mesh is quasi-uniform on the boundary (which condition was later removed by \cite{Gatica2012}).
While for our case, due to the surface-divergence-free boundary condition, we are able to follow the argument of \cite{Monk2003} 
without the quasi-uniform assumption.
The interpolation error is also indispensable in proving the convergence. 
The standard result concerning the edge element was provided by \cite{Alonso1999}, 
which requires that both the interpolated function and its curl belong to the Sobolev space with index greater than one half. 
However, since we demand certain uniformity of the interpolation error, it would be better if no additional regularity were assumed. 
Fortunately, \cite{Bermudez2005} pointed out that the regularity for the curl of the function can actually be weakened. 
Relying on this insight, we prove the error estimate when no regularity assumption is made. 
This interpolation result among others were collected in \cite{Ciarlet2016}.
We refer the readers to  \cite{Hiptmair2002finite} and \cite{Monk2003} for comprehensive surveys on the edge elements.

To prove the error estimate for the eigenvalue problem, we follow the classical approach \cite{Babuska1991, Boffi2010, Sun2016}. 
First, we show the discrete solution operator of the source problem converges in norm to the continuous one.
Second, we estimate the convergence order of the eigenvalues by the Babu\v{s}ka-Osborn theory. 
An interesting phenomenon appears when considering the convergence of the solution operator. 
We can prove, without much efforts, that the surface-divergence-free part of the boundary error holds the same order as the error in the curl norm. 
In contrast, the boundary error will normally miss a half order. 
This fact was observed in different circumstances, e.g., \cite{Liu2019, You2019}. 
In proving the convergence order of the eigenvalue, we propose a discrete eigenvalue problem which is different from the one in \cite{Camano2017} 
such that we can directly apply the Babu\v{s}ka-Osborn theory (see \cite{Halla2019approximation} for a different framework). 
Both formulations, ours and the  one in \cite{Camano2017}, stem from the same continuous eigenvalue problem
and provide similar numerical results. Unfortunately, we are not able to prove the method in \cite{Camano2017}.

The modified Maxwell's Stekloff eigenvalue problem is quite new. We provide a rigorous analysis for the source problem and the eigenvalue problem 
on Lipschitz polyhedra. Then a finite element method is proposed and the error estimates are proved.
The rest of the paper is arranged as follows. In Section 2 we introduce the definitions of various Sobolev spaces and 
show the well-posedness of the Maxwell's equation with surface-divergence-free Neumann data, which is the source problem 
associated to the modified Maxwell's Stekloff eigenvalue problem.
In Section 3, we prove the discrete compactness property and obtain the error estimate for the source problem.
Section 4 contains the eigenvalue problem and its finite element approximation. Using the Babu\v{s}ka-Osborn theory, we obtain the convergence order for the eigenvalues. 
In Section 5, numerical examples are presented. We make some conclusions and discuss future work in Section 6.

\section{The Source Problem}
We first introduce some preliminaries and refer the readers to \cite{Buffa2000, Hiptmair2002finite, Monk2003} for details.
Let $\Omega \subset \mathbb{R}^3$ be a simply connected bounded Lipschitz polyhedron with boundary $\Gamma$. 
Let $\boldsymbol{\nu}$ be the unit outer normal vector of $\Gamma$. 
Denote by $H^s(\Omega)$ and $H^t(\Gamma)$ the standard complex valued Sobolev spaces on $\Omega$ and $\Gamma$ for $s\in \mathbb{R}$ and $t\in [-1,1]$, respectively. 
Define
\begin{align*}
&\boldsymbol{H}^s(\Omega) := \big(H^s(\Omega)\big)^3,\; \boldsymbol{H}^t(\Gamma) := \big(H^t(\Gamma)\big)^3,\; \boldsymbol{L}^2(\Omega) := \big(L^2(\Omega)\big)^3,\; \boldsymbol{L}^2(\Gamma) := \big(L^2(\Gamma)\big)^3,
\\
&\boldsymbol{H}(\textbf{curl};\Omega) := \{\boldsymbol{u}\in \boldsymbol{L}^2(\Omega)\,| \,\textbf{curl}\,\boldsymbol{u}\in \boldsymbol{L}^2(\Omega)\},
\\
&\boldsymbol{H}_0(\textbf{curl};\Omega) := \{\boldsymbol{u}\in \boldsymbol{H}(\textbf{curl};\Omega)\,| \,\boldsymbol{\nu}\times\boldsymbol{u} = \boldsymbol{0}\; \text{on}\; \Gamma\},
\\
&\boldsymbol{H}(\text{div};\Omega) := \{\boldsymbol{u}\in \boldsymbol{L}^2(\Omega)\,| \,\text{div}\,\boldsymbol{u}\in L^2(\Omega)\},\\
&\boldsymbol{H}_0(\text{div};\Omega) := \{\boldsymbol{u}\in \boldsymbol{H}(\text{div};\Omega)\,| \,\boldsymbol{\nu}\cdot\boldsymbol{u} = 0\; \text{on}\; \Gamma\},
\\
&\boldsymbol{L}^2_t(\Gamma) := \{\boldsymbol{\mu}\in \boldsymbol{L}^2(\Gamma)\,| \, \boldsymbol{\nu}\cdot\boldsymbol{\mu} = 0\;\; \text{a.e.\! on}\; \Gamma\}.
\end{align*}
We denote the norms of $\boldsymbol{H}^s(\Omega)$ and $\boldsymbol{H}^t(\Gamma)$ respectively by $\Vert\cdot\Vert_{s,\Omega}$ and $\Vert\cdot\Vert_{t,\Gamma}$, 
and equip $\boldsymbol{H}(\textbf{curl};\Omega)$ with the norm
$\Vert \boldsymbol{u}\Vert_{\textbf{curl},\Omega}^2 := \Vert\boldsymbol{u}\Vert_{0,\Omega}^2 + \Vert\textbf{curl}\,\boldsymbol{u}\Vert_{0,\Omega}^2$.

Denote by $\Gamma_j$, $j = 1,\dots,J$, the boundary faces of $\Omega$. 
For $\psi \in L^2(\Gamma)$, let $\psi_j = \psi |_{\Gamma_j}$. The space $H^t(\Gamma)$ for $t>1$ is defined as \cite{Buffa2002}
\begin{align*}
H^t(\Gamma) = \{\psi \in H^1(\Gamma)\,| \,\psi_j\in H^t(\Gamma_j)\}
\end{align*}
with $\Vert \psi \Vert_{t,\Gamma}^2 := \Vert \psi \Vert_{1,\Gamma}^2 + \sum_{j=1}^J\Vert \psi_j\Vert_{t,\Gamma_j}^2$. 

Let $\nabla_{\Gamma}$, $\text{div}_{\Gamma}$, $\textbf{curl}_{\Gamma}$ and $\text{curl}_{\Gamma}$ denote, respectively, the surface gradient, surface divergence, surface vector curl and surface scalar curl. 
See, for example, \cite{Buffa2000} for their definitions. The following adjoint relations hold for $\boldsymbol{\phi} \in \boldsymbol{L}^2_t(\Gamma)$ and $\psi\in H^1(\Gamma)$:
\begin{align*}
\langle \boldsymbol{\phi}, \nabla_{\Gamma}\, \psi\rangle = -\langle \text{div}_{\Gamma}\, \boldsymbol{\phi},\psi\rangle,\quad \langle \boldsymbol{\phi},\textbf{curl}_{\Gamma}\,\psi\rangle = \langle \text{curl}_{\Gamma}\,\boldsymbol{\phi},\psi\rangle.
\end{align*}
Define the surface-divergence-free space as
\begin{align*}
&\boldsymbol{H}(\text{div}_{\Gamma}^0;\Gamma) := \{\boldsymbol{\mu}\in \boldsymbol{L}_t^2(\Gamma)\,| \,\text{div}_{\Gamma}\,\boldsymbol{\mu} = 0\},
\end{align*}
and equip $\boldsymbol{H}(\text{div}_{\Gamma}^0;\Gamma)$ with the norm $\Vert\cdot\Vert_{0,\Gamma}$.
Let 
\[
\boldsymbol{\gamma}_t : \big(C^{\infty}(\overline{\Omega})\big)^3 \rightarrow \boldsymbol{L}^2_t(\Gamma) \text{ and } 
\boldsymbol{\gamma}_T : \big(C^{\infty}(\overline{\Omega})\big)^3 \rightarrow \boldsymbol{L}^2_t(\Gamma)
\]
be the tangential operators that map $\boldsymbol{v}$ to $\boldsymbol{\nu}\times\boldsymbol{v}|_{\Gamma}$ and 
$(\boldsymbol{\nu}\times\boldsymbol{v}|_{\Gamma})\times\boldsymbol{\nu}$, respectively. 
It is well-known that $\boldsymbol{\gamma}_t$ and $\boldsymbol{\gamma}_T$ can be continuously extended to $\boldsymbol{H}(\textbf{curl};\Omega)$. 
The images $\boldsymbol{\gamma}_t(\boldsymbol{H}(\textbf{curl};\Omega))$ and 
$\boldsymbol{\gamma}_T(\boldsymbol{H}(\textbf{curl};\Omega))$ are characterized in \cite{Buffa2000} as 
$\boldsymbol{H}^{-1/2}(\text{div}_{\Gamma};\Gamma)$ and $\boldsymbol{H}^{-1/2}(\text{curl}_{\Gamma};\Gamma)$:
\begin{align*}
&\boldsymbol{H}^{-1/2}(\text{div}_{\Gamma};\Gamma) := \{\boldsymbol{\mu}\in V^{\prime}_{T} \,| \, \text{div}_{\Gamma}\,\boldsymbol{\mu} \in H^{-1/2}(\Gamma)\},
\\
&\boldsymbol{H}^{-1/2}(\text{curl}_{\Gamma};\Gamma) := \{\boldsymbol{\mu}\in V^{\prime}_{t}\,| \, \text{curl}_{\Gamma}\, \boldsymbol{\mu}\in H^{-1/2}(\Gamma)\}.
\end{align*}
Here $V_t$ and $V_T$ denote the traces of $\boldsymbol{H}^1(\Omega)$ such that $V_t = \boldsymbol{\gamma}_t(\boldsymbol{H}^1(\Omega))$ and $V_T = \boldsymbol{\gamma}_T(\boldsymbol{H}^1(\Omega))$. 
The spaces $V_t^{\prime}$ and $V_T^{\prime}$ are, respectively, the duals of $V_t$ and $V_T$ with $\boldsymbol{L}^2_t(\Gamma)$ acting as the pivot space. 
For the characterization of these spaces on Lipschitz polyhedra, we refer the readers to \cite{Buffa2001a}. 
Denote by  $\boldsymbol{v}_T = \boldsymbol{\gamma}_T \boldsymbol{v}$ the tangential component of $\boldsymbol{v}$.
Two useful facts are $\boldsymbol{\gamma}_T(\nabla p) = \nabla_{\Gamma} p$ for $p\in H^{1}(\Omega)$ (see Proposition 3.6 of \cite{Buffa2000}) and $\text{curl}_{\Gamma}\,\boldsymbol{u}_T = \boldsymbol{\nu}\cdot \textbf{curl}\,\boldsymbol{u}$ for $\boldsymbol{u}\in \boldsymbol{H}(\textbf{curl};\Omega)$ (see (40) of \cite{Buffa2000}).

Consider the source problem associated with the modified Maxwell's Stekloff eigenvalue problem. Given $\boldsymbol{f}\in \boldsymbol{H}(\text{div}_{\Gamma}^0;\Gamma)$, find $\boldsymbol{u}\in \boldsymbol{H}(\textbf{curl};\Omega)$ such that
\begin{align}
(\textbf{curl}\,\boldsymbol{u},\textbf{curl}\,\boldsymbol{v}) - \kappa^2(\epsilon_r \boldsymbol{u},\boldsymbol{v}) = \langle\boldsymbol{f},\boldsymbol{v}_T\rangle,\qquad\forall \boldsymbol{v}\in \boldsymbol{H}(\textbf{curl};\Omega).
\label{E:u}
\end{align}
Here $\kappa$ is the wavenumber which is real and positive and $\epsilon_r$ is the relative permittivity. 
Assume that the media is isotropic and dialectic, i.e., $\epsilon_r$ is a positive scalar function. 
In addition, we require that $\epsilon_r$ is piecewise smooth and bounded below. 
More precisely, suppose that there is a partition $\{\Omega_m\}_{m=1}^M$ of $\Omega$ satisfying 
$\overline{\Omega} = \cup_{m=1}^M\overline{\Omega}_m$, $\Omega_m\cap \Omega_n = \emptyset$ when $m\neq n$, 
and each subdomain $\Omega_m$ is connected and has a Lipschitz boundary. 
There exists a constant $\alpha > 0$ such that $\epsilon_r \in C^1(\overline{\Omega}_m)$ and $\epsilon_r \geqslant \alpha$.

For $\boldsymbol{u}, \boldsymbol{v}\in \boldsymbol{L}^2(\Omega)$ and $\boldsymbol{f}, \boldsymbol{g}\in \boldsymbol{L}^2_t(\Gamma)$, define
\begin{align*}
(\boldsymbol{u},\boldsymbol{v}) = \int_{\Omega} \boldsymbol{u}(\boldsymbol{x})\cdot \overline{\boldsymbol{v}(\boldsymbol{x})} dV(\boldsymbol{x})
\qquad \text{and}
\qquad \langle \boldsymbol{f},\boldsymbol{g}\rangle = \int_{\Gamma} \boldsymbol{f}(\boldsymbol{x})\cdot\overline{\boldsymbol{g}(\boldsymbol{x})} dA(\boldsymbol{x}).
\end{align*}
We also use $\langle\cdot,\cdot\rangle$ to denote the duality on the boundary between $\boldsymbol{H}^{-1/2}(\text{div}_{\Gamma};\Gamma)$ 
and $\boldsymbol{H}^{-1/2}(\text{curl}_{\Gamma};\Gamma)$, which is the case of  $\langle \boldsymbol{f},\boldsymbol{v}_T\rangle$ in \eqref{E:u}. 
The inclusion $\boldsymbol{H}(\text{div}_{\Gamma}^0;\Gamma) \subset \boldsymbol{H}^{-1/2}(\text{div}_{\Gamma};\Gamma)$ holds and
thus the right-hand-side of \eqref{E:u} is well-defined. Note that $\langle \boldsymbol{\mu},\boldsymbol{\zeta}\rangle$ regarded as a duality between
$\boldsymbol{H}^{-1/2}(\text{div}_{\Gamma};\Gamma)$ and $\boldsymbol{H}^{-1/2}(\text{curl}_{\Gamma};\Gamma)$ and $\langle \boldsymbol{\mu},\boldsymbol{\zeta}\rangle$ 
regarded as an inner product in $\boldsymbol{L}^2_t(\Gamma)$ coincide when 
$\boldsymbol{\mu}\in \boldsymbol{L}^2_t(\Gamma)\cap \boldsymbol{H}^{-1/2}(\text{div}_{\Gamma},\Gamma)$ and $\boldsymbol{\zeta}\in \boldsymbol{L}^2_t(\Gamma)\cap\boldsymbol{H}^{-1/2}(\text{curl}_{\Gamma},\Gamma)$.

Define a subset $\boldsymbol{\mathcal{Z}}(\Omega)$ of $\boldsymbol{H}(\textbf{curl};\Omega)$ by
\begin{align*}
\boldsymbol{\mathcal{Z}}(\Omega) = \{\boldsymbol{u}\in \boldsymbol{H}(\textbf{curl};\Omega)\;|\; b(\boldsymbol{u},q) = 0,\quad \forall q\in H^1(\Omega)\},
\end{align*}
where $b(\cdot,\cdot) : \boldsymbol{H}(\textbf{curl};\Omega)\times H^1(\Omega) \rightarrow \mathbb{C}$ is defined by
\begin{align*}
b(\boldsymbol{u},q) = (\epsilon_r \boldsymbol{u},\nabla q).
\end{align*}
Equip $\boldsymbol{\mathcal{Z}}(\Omega)$ with the norm $\Vert \cdot\Vert_{\textbf{curl},\Omega}$.
The following two lemmas are on the decomposition of $\boldsymbol{H}(\textbf{curl};\Omega)$ and the regularity of functions in $\boldsymbol{\mathcal{Z}}(\Omega)$.
\begin{lemma}
\label{L:decomp}
The space $\boldsymbol{H}(\normalfont{\textbf{curl}};\Omega)$ can be decomposed as
\begin{align*}
\boldsymbol{H}(\normalfont{\textbf{curl}};\Omega) = \boldsymbol{\mathcal{Z}}(\Omega) \oplus \nabla (H^1(\Omega)/\mathbb{C}).
\end{align*}
\end{lemma}
\begin{proof} 
By the Poincar\'{e}'s inequality, $\nabla(H^1(\Omega)/\mathbb{C})$ is a closed subspace of $\boldsymbol{H}(\textbf{curl};\Omega)$.
In addition, it is easily seen that $b(\boldsymbol{u},p)$ is an inner product of $\boldsymbol{u}$ and $\nabla p$ for $\boldsymbol{u}\in \boldsymbol{H}(\textbf{curl};\Omega)$ and $p\in H^1(\Omega)/\mathbb{C}$.
Thus $\boldsymbol{\mathcal{Z}}(\Omega)$ is the orthogonal complement of $\nabla(H^1(\Omega)/\mathbb{C})$. Then the decomposition is a direct consequence of the projection theorem.
\end{proof}

\begin{lemma}
\label{L:regul}
For $\boldsymbol{u}\in \boldsymbol{\mathcal{Z}}(\Omega)$, there exists $s_1, 0 < s_1 \leqslant 1/2$, such that for $0 \leqslant s < s_1$, 
it holds that $\boldsymbol{u}\in \boldsymbol{H}^{1/2+s}(\Omega)$ and $\boldsymbol{u}_T \in \boldsymbol{L}^{2}_t(\Gamma)$. Furthermore,
\begin{align*}
\Vert\boldsymbol{u}\Vert_{1/2+s,\Omega} \leqslant C\Vert \boldsymbol{u}\Vert_{\normalfont{\textbf{curl}},\Omega}\qquad\text{and}\qquad \Vert \boldsymbol{u}_T\Vert_{0,\Gamma}\leqslant C\Vert \boldsymbol{u}\Vert_{\normalfont{\textbf{curl}},\Omega}.
\end{align*}
\end{lemma}
\begin{proof}
Since $\boldsymbol{u}\in \boldsymbol{\mathcal{Z}}(\Omega)$, we have
$(\epsilon_r\boldsymbol{u},\nabla p) = 0$ for all $p \in H^1(\Omega)$. Hence $\text{div}(\epsilon_r\boldsymbol{u}) = 0$. 
Therefore $\text{div}\,\boldsymbol{u} = -\epsilon_r^{-1}\nabla\epsilon_r\cdot \boldsymbol{u}$,
which yields $\Vert \text{div}\,\boldsymbol{u}\Vert_{0,\Omega} \leqslant C\Vert \boldsymbol{u}\Vert_{0,\Omega}$. On the other hand, due to the Green's formula
\begin{align*}
(\epsilon_r\boldsymbol{u},\nabla p) + (\text{div}(\epsilon_r\boldsymbol{u}),p) = \langle \boldsymbol{\nu}\cdot(\epsilon_r\boldsymbol{u}),p\rangle,\qquad \forall p\in H^1(\Omega),
\end{align*}
it holds that $\boldsymbol{\nu}\cdot(\epsilon_r\boldsymbol{u}) = 0$. Thus $\boldsymbol{\nu}\cdot\boldsymbol{u} = 0$
and $\boldsymbol{u}\in \boldsymbol{H}(\textbf{curl};\Omega) \cap \boldsymbol{H}_0(\text{div};\Omega)$. 
By Proposition 3.7 of \cite{Amrouche1998} (see also Theorem 3.50 of \cite{Monk2003}), there exists $s_{\Omega} > 0$ 
such that for all $0 \leqslant s < s_1 := \min(s_{\Omega},1/2)$, $\boldsymbol{u}\in \boldsymbol{H}^{1/2+s}(\Omega)$ and
\begin{align*}
\Vert\boldsymbol{u}\Vert_{1/2+s,\Omega} \leqslant C\Big(\Vert \textbf{curl}\,\boldsymbol{u}\Vert_{0,\Omega} + \Vert \text{div}\,\boldsymbol{u}\Vert_{0,\Omega} + \Vert \boldsymbol{u}\Vert_{0,\Omega}\Big) \leqslant C\Vert \boldsymbol{u}\Vert_{\textbf{curl},\Omega}.
\end{align*}
Due to the fact that $\boldsymbol{\nu}\cdot\boldsymbol{u} = 0$ and the trace theorem, 
$\boldsymbol{u}_T \in \boldsymbol{L}^2_t(\Gamma)$ and $\Vert\boldsymbol{u}_T\Vert_{0,\Gamma} \leqslant C\Vert\boldsymbol{u}\Vert_{1/2+s,\Omega}$. The proof is complete.
\end{proof}

Now we study the well-posedness of the source problem \eqref{E:u}. Define the sesquilinear forms 
$a : \boldsymbol{H}(\textbf{curl};\Omega)\times\boldsymbol{H}(\textbf{curl};\Omega) \rightarrow \mathbb{C}$:
\[
a(\boldsymbol{u},\boldsymbol{v}) = (\textbf{curl}\,\boldsymbol{u},\textbf{curl}\,\boldsymbol{v}) - \kappa^2(\epsilon_r\boldsymbol{u},\boldsymbol{v}),
\]
and $a_+ : \boldsymbol{H}(\textbf{curl};\Omega)\times\boldsymbol{H}(\textbf{curl};\Omega) \rightarrow \mathbb{C}$:
\[
a_+(\boldsymbol{u},\boldsymbol{v}) = (\textbf{curl}\,\boldsymbol{u},\textbf{curl}\,\boldsymbol{v}) + (\epsilon_r\boldsymbol{u},\boldsymbol{v}).
\]
We shall relate \eqref{E:u} to the problem of 
finding $\boldsymbol{u}\in \boldsymbol{L}^2(\Omega)$ such that
\begin{align}
(I+K)\boldsymbol{u} = \boldsymbol{z},
\label{E:u_K}
\end{align}
where $\boldsymbol{z}\in \boldsymbol{\mathcal{Z}}(\Omega)$ satisfies
\begin{align}
a_+(\boldsymbol{z},\boldsymbol{v}) = \langle\boldsymbol{f},\boldsymbol{v}_T\rangle,\qquad \forall \boldsymbol{v}\in \boldsymbol{\mathcal{Z}}(\Omega),
\label{E:z_Z}
\end{align}
and 
$K : \boldsymbol{L}^2(\Omega)\rightarrow \boldsymbol{\mathcal{Z}}(\Omega) \subset\subset \boldsymbol{L}^2(\Omega)$ is such that
\begin{align}
a_+(K\boldsymbol{g},\boldsymbol{v}) = -(\kappa^2+1)(\epsilon_r\boldsymbol{g},\boldsymbol{v}),\qquad \forall \boldsymbol{v}\in \boldsymbol{\mathcal{Z}}(\Omega).
\label{E:Kg_Z}
\end{align}
The following lemma states that \eqref{E:z_Z} and \eqref{E:Kg_Z} are well-posed.
\begin{lemma}
\label{L:z,Ku}
For $\boldsymbol{f}\in \boldsymbol{H}(\normalfont{\text{div}}_{\Gamma}^0;\Gamma)$ and $\boldsymbol{g}\in \boldsymbol{L}^2(\Omega)$,
there exist a unique solution $\boldsymbol{z}\in \boldsymbol{\mathcal{Z}}(\Omega)$ of \eqref{E:z_Z} and 
a unique solution $K\boldsymbol{g}\in \boldsymbol{\mathcal{Z}}(\Omega)$ of \eqref{E:Kg_Z}. Furthermore,
\begin{align*}
\Vert\boldsymbol{z}\Vert_{\normalfont{\textbf{curl}},\Omega} \leqslant C\Vert \boldsymbol{f}\Vert_{0,\Gamma}, \qquad \Vert K\boldsymbol{g}\Vert_{\normalfont{\textbf{curl}},\Omega}\leqslant C\Vert \boldsymbol{g}\Vert_{0,\Omega}.
\end{align*} 
\end{lemma}
\begin{proof}
It is clear that $a_+(\cdot,\cdot)$ is coercive on $\boldsymbol{\mathcal{Z}}(\Omega)$ and bounded on $\boldsymbol{\mathcal{Z}}(\Omega)\times\boldsymbol{\mathcal{Z}}(\Omega)$.
Moreover, the right-hand-side of \eqref{E:z_Z}  and \eqref{E:Kg_Z} are bounded with respect to $\boldsymbol{v}$, i.e.,
\begin{align*}
\vert \langle\boldsymbol{f},\boldsymbol{v}_T\rangle\vert \leqslant \Vert\boldsymbol{f}\Vert_{0,\Gamma} \Vert\boldsymbol{v}_T\Vert_{0,\Gamma}\leqslant C\Vert \boldsymbol{f}\Vert_{0,\Gamma}\Vert \boldsymbol{v}\Vert_{\textbf{curl},\Omega}\quad\text{and}\quad \vert\! -\!(\kappa^2+1)(\epsilon_r\boldsymbol{g},\boldsymbol{v})\vert\leqslant C\Vert\boldsymbol{g}\Vert_{0,\Omega}\Vert\boldsymbol{v}\Vert_{\textbf{curl},\Omega}.
\end{align*}
Therefore the uniqueness, existence and the continuous dependence hold for $\boldsymbol{z}$ and $K\boldsymbol{g}$.
\end{proof}

The equivalence of \eqref{E:u_K} and \eqref{E:u} is shown in the following lemma.
\begin{lemma}
\label{L:equiv}
Given $\boldsymbol{f}\in \boldsymbol{H}(\normalfont{\text{div}}_{\Gamma}^0;\Gamma)$, $\boldsymbol{u}\in \boldsymbol{H}(\normalfont{\textbf{curl}};\Omega)$ is a solution of \eqref{E:u} if and only if $\boldsymbol{u}\in \boldsymbol{L}^2(\Omega)$ is a solution of \eqref{E:u_K}.
\end{lemma}
\begin{proof}
Let $\boldsymbol{f}\in \boldsymbol{H}(\text{div}^0_{\Gamma};\Gamma)$.
If $\boldsymbol{u}\in \boldsymbol{H}(\textbf{curl};\Omega)$ is a solution of \eqref{E:u}, 
letting $\boldsymbol{v} = \nabla p$ in \eqref{E:u}, we have $(\epsilon_r\boldsymbol{u},\nabla p) = 0$ for $p\in H^1(\Omega)$. 
Thus $\boldsymbol{u}\in \boldsymbol{\mathcal{Z}}(\Omega)$.
Letting $\boldsymbol{v}\in \boldsymbol{\mathcal{Z}}(\Omega)$ in \eqref{E:u}, we obtain that
\begin{align*}
a_+(\boldsymbol{u},\boldsymbol{v}) + a_+(K\boldsymbol{u},\boldsymbol{v}) = a_+(\boldsymbol{z},\boldsymbol{v}),\qquad \forall \boldsymbol{v}\in \boldsymbol{\mathcal{Z}}(\Omega).
\end{align*}
By the coercivity of $a_+(\cdot,\cdot)$, $\boldsymbol{u}$ satisfies \eqref{E:u_K}. 

Conversely, if $\boldsymbol{u}\in \boldsymbol{L}^2(\Omega)$ is a solution of \eqref{E:u_K}, then it holds that $\boldsymbol{u} = \boldsymbol{z} - K\boldsymbol{u}$. 
This implies that $\boldsymbol{u}$ actually belongs to $\boldsymbol{\mathcal{Z}}(\Omega)$. 
For all $\boldsymbol{v}\in \boldsymbol{H}(\textbf{curl};\Omega)$, due to Lemma \ref{L:decomp}, 
there exist $\boldsymbol{v}_0\in \boldsymbol{\mathcal{Z}}(\Omega)$ and $p\in H^1(\Omega)$ such that $\boldsymbol{v} = \boldsymbol{v}_0 + \nabla p$. 
Since $\boldsymbol{u}$ solves \eqref{E:u_K}, we have that
\begin{align*}
a(\boldsymbol{u},\boldsymbol{v}) = a(\boldsymbol{u},\boldsymbol{v}_0) = a_+((I+K)\boldsymbol{u},\boldsymbol{v}_0) = a_+(\boldsymbol{z},\boldsymbol{v}_0) = \langle\boldsymbol{f},(\boldsymbol{v}_0)_T\rangle = \langle \boldsymbol{f},\boldsymbol{v}_T\rangle,\quad \forall \boldsymbol{v}\in \boldsymbol{H}(\textbf{curl};\Omega),
\end{align*}
i.e., $\boldsymbol{u}$ is a solution of \eqref{E:u}.
\end{proof}  
\begin{definition} We call $\kappa^2$ a Neumann eigenvalue of the Maxwell's equation if there exists a non-trivial function $\boldsymbol{u}$ such that
\begin{align*}
{\normalfont\textbf{curl}}\,{\normalfont\textbf{curl}}\,\boldsymbol{u} - \kappa^2\epsilon_r\boldsymbol{u} &= \boldsymbol{0},\qquad\quad \text{in}\; \Omega,
\\
\boldsymbol{\nu}\times{\normalfont\textbf{curl}}\,\boldsymbol{u} &= \boldsymbol{0},\qquad\quad \text{on}\;\Gamma.
\end{align*}
\end{definition}
In the rest of the paper, we assume that $\kappa^2$ is not a Neumann eigenvalue.
The following lemma shows the well-posedness for the source problem.
\begin{lemma}
\label{L:K}
The operator $K : \boldsymbol{L}^2(\Omega)\rightarrow \boldsymbol{L}^2(\Omega)$ is compact. 
Given $\boldsymbol{z}\in \boldsymbol{L}^2(\Omega)$, there exists a unique solution $\boldsymbol{u}\in \boldsymbol{L}^2(\Omega)$ of \eqref{E:u_K}, 
which depends continuously on $\boldsymbol{z}$. Furthermore, \eqref{E:u} has a unique solution $\boldsymbol{u}\in \boldsymbol{\mathcal{Z}}(\Omega)$ satisfying
\begin{align}\label{ucurlf0}
\Vert\boldsymbol{u}\Vert_{\normalfont{\textbf{curl}},\Omega} \leqslant C\Vert\boldsymbol{f}\Vert_{0,\Gamma}.
\end{align}
\end{lemma}
\begin{proof}
By Lemma \ref{L:regul}, $K$ is a continuous operator from $\boldsymbol{L}^2(\Omega)$ to $\boldsymbol{H}^{1/2+s}(\Omega)$, 
which is compactly embedded in $\boldsymbol{L}^2(\Omega)$. Hence $K$ is compact. 
Since $\kappa^2$ is not a Neumann eigenvalue, we have the uniqueness for \eqref{E:u}.
By Lemma \ref{L:equiv}, the uniqueness also holds for \eqref{E:u_K}.
Then the Fredholm alternative ensures the existence of a unique solution $\boldsymbol{u}\in \boldsymbol{L}^2(\Omega)$ of \eqref{E:u_K} and
\begin{align}
\Vert \boldsymbol{u}\Vert_{0,\Omega}\leqslant C\Vert \boldsymbol{z}\Vert_{0,\Omega}.
\label{IE:u_0}
\end{align}
By Lemma \ref{L:equiv}, $\boldsymbol{u}$ is also the solution of \eqref{E:u}. 
Taking
$\boldsymbol{v} = \boldsymbol{u}$ in \eqref{E:u} and recalling that $\boldsymbol{u}\in \boldsymbol{\mathcal{Z}}(\Omega)$, we have that
\begin{align*}
\Vert\boldsymbol{u}\Vert_{\textbf{curl},\Omega}^2 - C\Vert\boldsymbol{u}\Vert_{0,\Omega}^2 \leqslant a(\boldsymbol{u},\boldsymbol{u}) = \langle\boldsymbol{f},\boldsymbol{u}_T\rangle \leqslant \Vert\boldsymbol{f}\Vert_{0,\Gamma}\Vert\boldsymbol{u}\Vert_{\textbf{curl},\Omega}.
\end{align*}
Using the above inequality and \eqref{IE:u_0}, we obtain \eqref{ucurlf0}.
\end{proof}

\begin{lemma}
\label{L:regul_Ku,z}
Let $\boldsymbol{z}\in \boldsymbol{\mathcal{Z}}(\Omega)$ be the solution of \eqref{E:z_Z} for $\boldsymbol{f}\in \boldsymbol{H}(\normalfont{\text{div}}_{\Gamma}^0;\Gamma)$
and $\boldsymbol{u}\in \boldsymbol{\mathcal{Z}}(\Omega)$ be the solution of \eqref{E:u_K}. There exists $s_2$, $0<s_2\leqslant 1/2$, such that, for $0\leqslant s < \min\{s_1,s_2\}$,
\begin{align*}
\Vert K\boldsymbol{u}\Vert_{1/2+s,\Omega} + \Vert \normalfont{\textbf{curl}}\,K\boldsymbol{u}\Vert_{1/2+s,\Omega}  \leqslant C \Vert \boldsymbol{f}\Vert_{0,\Gamma}.
\end{align*}
Furthermore, $\normalfont{\textbf{curl}}\,\boldsymbol{u}\in \boldsymbol{H}^{1/2}(\Omega)$ with $\Vert \normalfont{\textbf{curl}}\,\boldsymbol{u}\Vert_{1/2,\Omega}\leqslant C\Vert \boldsymbol{f}\Vert_{0,\Gamma}$ and $\normalfont{\textbf{curl}}\,\boldsymbol{z}\in \boldsymbol{H}^{1/2}(\Omega)$ with
\begin{align*}
\Vert \boldsymbol{z}\Vert_{1/2+s,\Omega} + \Vert \normalfont{\textbf{curl}}\,\boldsymbol{z}\Vert_{1/2,\Omega}\leqslant C\Vert \boldsymbol{f}\Vert_{0,\Gamma},
\end{align*}
where $0\leqslant s< s_1$.
\end{lemma}
\begin{proof}
Since $K\boldsymbol{u}\in \boldsymbol{\mathcal{Z}}(\Omega)$, by Lemmas \ref{L:regul}, \ref{L:z,Ku} and \ref{L:K}, we have for $0\leqslant s< s_1$ that
\begin{align*}
\Vert K\boldsymbol{u}\Vert_{1/2+s,\Omega}\leqslant C\Vert K\boldsymbol{u}\Vert_{\textbf{curl},\Omega} \leqslant C\Vert \boldsymbol{u}\Vert_{0,\Omega}\leqslant C\Vert \boldsymbol{f}\Vert_{0,\Gamma}.
\end{align*}
Given $\boldsymbol{v}\in \boldsymbol{H}(\textbf{curl};\Omega)$, let $\boldsymbol{v}_0\in \boldsymbol{\mathcal{Z}}(\Omega)$ and $p\in H^1(\Omega)/\mathbb{C}$ 
be such that $\boldsymbol{v} = \boldsymbol{v}_0 + \nabla p$ due to Lemma~\ref{L:decomp}. From \eqref{E:Kg_Z}, $K\boldsymbol{u}$ satisfies
\begin{align}
a_+(K\boldsymbol{u},\boldsymbol{v}) = -(\kappa^2+1)(\epsilon_r\boldsymbol{u},\boldsymbol{v}),\qquad \forall \boldsymbol{v}\in \boldsymbol{H}(\textbf{curl};\Omega).
\label{E:Ku}
\end{align}
By the Green's formula, $\boldsymbol{\nu}\times\textbf{curl}\,K\boldsymbol{u} = \boldsymbol{0}$.
Consequently, $\textbf{curl}\,K\boldsymbol{u}\in \boldsymbol{H}_0(\textbf{curl};\Omega)\cap\boldsymbol{H}(\text{div};\Omega)$. 
By Proposition 3.7 of \cite{Amrouche1998} there is $s_{\Omega}^0 > 0$ such that $\textbf{curl}\,K\boldsymbol{u}\in \boldsymbol{H}^{1/2+s}(\Omega)$ 
for $0\leqslant s < s_2 := \min(s_{\Omega}^0,1/2)$. Furthermore,
\begin{align*}
\Vert \textbf{curl}\, K\boldsymbol{u}\Vert_{1/2+s,\Omega} & \leqslant C\Big(\Vert \textbf{curl}\,K\boldsymbol{u}\Vert_{0,\Omega} + \Vert \textbf{curl}\,\textbf{curl}\,K\boldsymbol{u}\Vert_{0,\Omega}\Big)\\
&\leqslant C\Big(\Vert K\boldsymbol{u}\Vert_{\textbf{curl},\Omega} + \Vert \boldsymbol{u}\Vert_{0,\Omega}\Big) 
\leqslant C\Vert\boldsymbol{f}\Vert_{0,\Gamma}.
\end{align*}
For the term $\textbf{curl}\,\boldsymbol{u}$, we apply the regularity result in \cite{Costabel1990} to obtain that
\begin{align*}
\Vert \textbf{curl}\,\boldsymbol{u}\Vert_{1/2,\Omega}\leqslant C\Big(\Vert \textbf{curl}\,\boldsymbol{u}\Vert_{0,\Omega} + \Vert\textbf{curl}\,\textbf{curl}\,\boldsymbol{u}\Vert_{0,\Omega} + \Vert\boldsymbol{\nu}\times\textbf{curl}\,\boldsymbol{u}\Vert_{0,\Gamma}\Big)\leqslant C\Vert\boldsymbol{f}\Vert_{0,\Gamma}.
\end{align*}
Using the previous results, it holds for $0\leqslant s < s_1$ that
\begin{align*}
\Vert \boldsymbol{z}\Vert_{1/2+s,\Omega} + \Vert \normalfont{\textbf{curl}}\,\boldsymbol{z}\Vert_{1/2,\Omega}&\leqslant C\Big(\Vert\boldsymbol{z}\Vert_{\textbf{curl},\Omega} + \Vert \textbf{curl}\,\boldsymbol{u}\Vert_{1/2,\Omega} + \Vert\textbf{curl}\, K\boldsymbol{u}\Vert_{1/2,\Omega}\Big)\leqslant C\Vert \boldsymbol{f}\Vert_{0,\Gamma},
\end{align*}
where we have used $\textbf{curl}\,\boldsymbol{z} = \textbf{curl}\,\boldsymbol{u} + \textbf{curl}\,K\boldsymbol{u}$.
\end{proof}
In fact, given $\boldsymbol{f}\in \boldsymbol{H}(\text{div}^0_{\Gamma};\Gamma)$, \eqref{E:z_Z} and \eqref{E:Kg_Z} 
can also be defined on $\boldsymbol{H}(\textbf{curl};\Omega)$.
\begin{lemma}
\label{L:equiv2}
Given $\boldsymbol{f}\in \boldsymbol{H}(\normalfont{\text{div}}^0_{\Gamma};\Gamma)$ and $\boldsymbol{g}\in \boldsymbol{L}^2(\Omega)$, 
there exist a unique solution $\boldsymbol{z}\in \boldsymbol{\mathcal{Z}}(\Omega)$ of
\begin{align}
a_+(\boldsymbol{z},\boldsymbol{v}) = \langle \boldsymbol{f},\boldsymbol{v}_T\rangle,\qquad \forall \boldsymbol{v}\in \boldsymbol{H}(\normalfont{\textbf{curl}};\Omega),
\label{E:z}
\end{align}
and a unique solution $(K\boldsymbol{g},\phi)\in \boldsymbol{\mathcal{Z}}(\Omega)\times H^1(\Omega)/\mathbb{C}$ of
\begin{align}
a_+(K\boldsymbol{g},\boldsymbol{v}) + b(\boldsymbol{v},\phi) &= -(\kappa^2+1)(\epsilon_r\boldsymbol{g},\boldsymbol{v}),\qquad \forall \boldsymbol{v}\in \boldsymbol{H}(\normalfont{\textbf{curl}};\Omega),
\label{E:Kg1}
\\
b(K\boldsymbol{g},q) &= 0,\qquad\qquad\qquad\qquad\quad \forall q\in H^1(\Omega)/\mathbb{C}.
\label{E:Kg2}
\end{align}
Furthermore, $\boldsymbol{z}$ is the solution of \eqref{E:z_Z} and $K\boldsymbol{g}$ is the solution of \eqref{E:Kg_Z}.
\end{lemma}
\begin{proof}
The right-hand-side of \eqref{E:z_Z} is a continuous linear functional of $\boldsymbol{v}\in \boldsymbol{H}(\textbf{curl};\Omega)$. 
Then by the coercivity and boundedness of $a_+(\cdot,\cdot)$, there exists a unique solution of \eqref{E:z}. 
For any $p\in H^1(\Omega)$, taking $\boldsymbol{v} = \nabla p$ in \eqref{E:z} yields $\boldsymbol{z}\in \boldsymbol{\mathcal{Z}}(\Omega)$. 
Thus $\boldsymbol{z}$ is also the solution of \eqref{E:z_Z}. 

It is clear that $a_+(\cdot,\cdot)$ is $\boldsymbol{\mathcal{Z}}(\Omega)$-coercive and the Babu\v{s}ka-Brezzi condition holds due to
\begin{align*}
\sup_{\boldsymbol{v}\in \boldsymbol{H}(\textbf{curl};\Omega)}\frac{|b(\boldsymbol{v},q)|}{\Vert\boldsymbol{v}\Vert_{\textbf{curl},\Omega}} \geqslant \frac{|(\epsilon_r\nabla q,\nabla q)|}{\Vert\nabla q\Vert_{\textbf{curl},\Omega}} \geqslant C\Vert\nabla q\Vert_{0,\Omega} \geqslant C\Vert q\Vert_{H^1(\Omega)/\mathbb{C}},\qquad \forall q\in H^1(\Omega)/\mathbb{C}.
\end{align*}
Therefore, there is a unique solution $(K\boldsymbol{g},\phi)\in \boldsymbol{H}(\textbf{curl};\Omega)\times H^1(\Omega)/\mathbb{C}$ of \eqref{E:Kg1}-\eqref{E:Kg2}. 
Due to \eqref{E:Kg2}, $K\boldsymbol{g}\in \boldsymbol{\mathcal{Z}}(\Omega)$. 
Taking $\boldsymbol{v}\in \boldsymbol{\mathcal{Z}}(\Omega)$ in \eqref{E:Kg1}, $K\boldsymbol{g}$ satisfies \eqref{E:Kg_Z}.
\end{proof}

\section{Finite Element Method for the Source Problem}
In this section, we propose a finite element method for \eqref{E:u} and prove its convergence. 
Let $\tau_h$ be a regular tetrahedral mesh for $\Omega$ with size $h$. 
Since $\Omega$ is a polyhedron, the faces of $\tau_h$ on $\Gamma$ induce a triangular mesh for $\Gamma$.
We use the notations in Chapter 5 of \cite{Monk2003} to denote by $W_h\subset \boldsymbol{H}(\text{div},\Omega)$ 
the divergence-conforming finite element space of degree $k$, by $V_h\subset \boldsymbol{H}(\textbf{curl};\Omega)$ the curl-conforming finite element space of degree $k$, 
and by $U_h\subset H^1(\Omega)$ the $H^1$-conforming finite element space of degree $k$. 
We shall mainly discuss the case when $k = 1$, though the analysis extends to $k > 1$ if higher regularity of the solution is assumed. 

Denote by $\pi_h^1 : \boldsymbol{H}(\textbf{curl};\Omega)\supset \mathcal{V} \rightarrow V_h$ and 
$\pi_h^2 : \boldsymbol{H}(\text{div};\Omega) \supset \mathcal{W} \rightarrow W_h$ the interpolation operators. 
Here $\mathcal{V}$ and $\mathcal{W}$ are suitable subspaces such that the interpolations are well-defined and bounded (see, e.g., Lemma 5.38 of \cite{Monk2003}). 
The finite element spaces $W_h$, $V_h$ and $U_h$ satisfy the de Rham complex (see, e.g., (5.59) of  \cite{Monk2003}), which implies
\begin{align}
\textbf{curl}\,V_h \subset W_h,\qquad \nabla U_h \subset V_h, \qquad \text{and}\quad \textbf{curl}\,\pi_h^1\boldsymbol{v} = \pi_h^2\textbf{curl}\,\boldsymbol{v} 
\quad \text{ for } \boldsymbol{v}\in \mathcal{V}.
\label{E:de Rham}
\end{align} 
Following the definition of $\boldsymbol{\mathcal{Z}}(\Omega)$, define
\begin{align*}
\boldsymbol{\mathcal{Z}}_h = \{\boldsymbol{u}_h\in V_h\;|\; b(\boldsymbol{u}_h,p_h) = 0,\quad\forall p_h\in U_h\}.
\end{align*}
The discrete problem for \eqref{E:u} is to find $\boldsymbol{u}_h\in V_h$ such that
\begin{align}
a(\boldsymbol{u}_h,\boldsymbol{v}_h) = \langle \boldsymbol{f}_h,\boldsymbol{v}_{h,T}\rangle,\qquad \forall \boldsymbol{v}_h\in V_h,
\label{E:uh}
\end{align}
where $\boldsymbol{f}_h\in \boldsymbol{H}(\text{div}_{\Gamma}^0;\Gamma)$. 
We note that $\boldsymbol{f}_h$ could be taken as $\boldsymbol{f}$ in \eqref{E:u} or some approximation of $\boldsymbol{f}$.
Similar to the continuous counterpart, we transfer \eqref{E:uh} into finding $\boldsymbol{u}_h\in \boldsymbol{L}^2(\Omega)$ such that
\begin{align}
(I+K_h)\boldsymbol{u}_h = \boldsymbol{z}_h,
\label{E:uh_Kh}
\end{align}
where $\boldsymbol{z}_h\in \boldsymbol{\mathcal{Z}}_h$ satisfies
\begin{align}
a_+(\boldsymbol{z}_h,\boldsymbol{v}_h) = \langle\boldsymbol{f}_h,\boldsymbol{v}_{h,T}\rangle,\qquad \forall\boldsymbol{v}_h\in \boldsymbol{\mathcal{Z}}_h,
\label{E:zh_Zh}
\end{align}
and $K_h : \boldsymbol{L}^2(\Omega) \rightarrow \boldsymbol{\mathcal{Z}}_h\subset \boldsymbol{L}^2(\Omega)$ is such that
\begin{align}
a_+(K_h\boldsymbol{g},\boldsymbol{v}_h) = -(\kappa^2+1)(\epsilon_r\boldsymbol{g},\boldsymbol{v}_h),\qquad \forall \boldsymbol{v}_h\in \boldsymbol{\mathcal{Z}}_h.
\label{E:Kh g_Zh}
\end{align}
Using the same arguments as the proofs of Lemmas \ref{L:z,Ku} and \ref{L:equiv}, 
we can show the well-posedness of \eqref{E:zh_Zh} and \eqref{E:Kh g_Zh} as well as the equivalence of \eqref{E:uh} and \eqref{E:uh_Kh}.
\begin{lemma}
\label{L:equiv_h}
Given $\boldsymbol{f}_h\in \boldsymbol{H}(\normalfont{\text{div}}_{\Gamma}^0;\Gamma)$ and $\boldsymbol{g}\in \boldsymbol{L}^2(\Omega)$, 
there exist unique solutions $\boldsymbol{z}_h$ and $K_h\boldsymbol{g}$ of \eqref{E:zh_Zh} and \eqref{E:Kh g_Zh}, respectively. 
Furthermore, $\boldsymbol{u}_h\in V_h$ is a solution of \eqref{E:uh} if and only if $\boldsymbol{u}_h\in \boldsymbol{L}^2(\Omega)$ is a solution of \eqref{E:uh_Kh}.
\end{lemma}
\begin{proof}
The well-posedness of \eqref{E:zh_Zh} and \eqref{E:Kh g_Zh} follows the coercivity and boundedness of $a_+(\cdot,\cdot)$. 
In addition, we have
\begin{align*}
\Vert K_h\boldsymbol{g}\Vert_{\textbf{curl},\Omega}\leqslant C\Vert\boldsymbol{g}\Vert_{0,\Omega}
\end{align*}
with $C$ independent of $h$.
By \eqref{E:de Rham}, the finite dimensional space $\nabla(U_h/\mathbb{C})$ is a closed subspace of $V_h$. Therefore, we have the decomposition
\begin{align*}
V_h = \boldsymbol{\mathcal{Z}}_h \oplus \nabla(U_h/\mathbb{C}).
\end{align*}
Due to Lemma \ref{L:equiv}, \eqref{E:uh} is equivalent to  \eqref{E:uh_Kh}.
\end{proof}

Next we prove the well-posedness of \eqref{E:uh_Kh}. 
We first show that the finite element solutions of \eqref{E:zh_Zh} and \eqref{E:Kh g_Zh} approximate the solutions of \eqref{E:z_Z} and \eqref{E:Kg_Z}, respectively.
Similar to the equations \eqref{E:z} and \eqref{E:Kg1}-\eqref{E:Kg2}, consider the problems of finding $\boldsymbol{z}_h\in V_h$ such that
\begin{align}
a_+(\boldsymbol{z}_h,\boldsymbol{v}_h) = \langle \boldsymbol{f}_h,\boldsymbol{v}_{h,T}\rangle,\qquad \forall \boldsymbol{v}_h\in V_h,
\label{E:zh}
\end{align}
and $(K_h\boldsymbol{g},\phi_h)\in V_h \times U_h/\mathbb{C}$ such that
\begin{align}
a_+(K_h\boldsymbol{g},\boldsymbol{v}_h) + b(\boldsymbol{v}_h,\phi_h) &= -(\kappa^2+1)(\epsilon_r \boldsymbol{g},\boldsymbol{v}_h),\qquad \forall\boldsymbol{v}_h\in V_h,
\label{E:Kh g1}
\\
b(K_h\boldsymbol{g},q_h) &= 0,\qquad\qquad\qquad\qquad\quad \forall q_h \in U_h.
\label{E:Kh g2}
\end{align}
The next lemma claims the well-posedness of \eqref{E:zh} and \eqref{E:Kh g1}-\eqref{E:Kh g2}, 
the equivalence of \eqref{E:zh_Zh} and \eqref{E:zh}, and the equivalence of \eqref{E:Kh g_Zh} and \eqref{E:Kh g1}-\eqref{E:Kh g2}. 
In addition, the quasi-optimal error estimates of the finite element solution $\boldsymbol{z}_h$ and $K_h\boldsymbol{g}$ are obtained.
\begin{lemma}
\label{L:zh}
Given $\boldsymbol{f}_h\in \boldsymbol{H}(\normalfont{\text{div}}_{\Gamma}^0;\Gamma)$ and $\boldsymbol{g}\in \boldsymbol{L}^2(\Omega)$, 
there exist, respectively, a unique solution $\boldsymbol{z}_h\in V_h$ of \eqref{E:zh} and a unique solution 
$(K_h\boldsymbol{g},\phi_h)\in V_h\times U_h/\mathbb{C}$ of \eqref{E:Kh g1}-\eqref{E:Kh g2}. 
Furthermore, $\boldsymbol{z}_h\in \boldsymbol{\mathcal{Z}}_h$ and $K_h\boldsymbol{g}\in \boldsymbol{\mathcal{Z}}_h$ 
are the solutions of \eqref{E:zh_Zh} and \eqref{E:Kh g_Zh},  respectively. 
Given $\boldsymbol{f}\in \boldsymbol{H}(\normalfont{\text{div}}_{\Gamma}^0;\Gamma)$, 
if $\boldsymbol{z}\in \boldsymbol{\mathcal{Z}}(\Omega)$ and $K\boldsymbol{g}\in \boldsymbol{\mathcal{Z}}(\Omega)$ solve \eqref{E:z_Z} and \eqref{E:Kg_Z}, respectively,
then the following error estimates hold
\begin{align*}
\Vert \boldsymbol{z} - \boldsymbol{z}_h\Vert_{\normalfont{\textbf{curl}},\Omega} &\leqslant C\inf_{\boldsymbol{v}_h\in V_h}\Vert \boldsymbol{z} - \boldsymbol{v}_h\Vert_{\normalfont{\textbf{curl}},\Omega} + C\Vert \boldsymbol{f} - \boldsymbol{f}_h\Vert_{0,\Gamma},
\\
\Vert (K - K_h)\boldsymbol{g}\Vert_{\normalfont{\textbf{curl}},\Omega} &\leqslant C\inf_{\boldsymbol{v}_h\in V_h}\Vert K\boldsymbol{g} - \boldsymbol{v}_h\Vert_{\normalfont{\textbf{curl}},\Omega} + C\inf_{\psi_h\in U_h}\Vert \nabla \phi - \nabla \psi_h\Vert_{0,\Omega}.
\end{align*}
\end{lemma}
\begin{proof}
Noticing that $V_h$ and $U_h$ are conforming finite element spaces that satisfy \eqref{E:de Rham}, 
we can show the well-posedness and the equivalence similarly to Lemma \ref{L:equiv2}. 
The second error estimate follows Theorem 2.45 of \cite{Monk2003}. To show the estimate for $\boldsymbol{z} - \boldsymbol{z}_h$, 
let $\boldsymbol{z}^h\in \boldsymbol{\mathcal{Z}}$ be the solution of 
\begin{align*}
a_+(\boldsymbol{z}^h,\boldsymbol{v}) = \langle\boldsymbol{f}_h,\boldsymbol{v}_T\rangle,\qquad  \forall \boldsymbol{v}\in \boldsymbol{H}(\textbf{curl};\Omega).
\end{align*}
By the well-posedness of \eqref{E:z}, it holds that
\begin{align*}
\Vert \boldsymbol{z} - \boldsymbol{z}^h\Vert_{\textbf{curl},\Omega} \leqslant C \Vert \boldsymbol{f} - \boldsymbol{f}_h\Vert_{0,\Gamma}.
\end{align*} 
On the other hand, using the Cea's lemma, we have that
\begin{align*}
\Vert \boldsymbol{z}^h - \boldsymbol{z}_h\Vert_{\textbf{curl},\Omega}\leqslant C \inf_{\boldsymbol{v}_h\in V_h}\Vert \boldsymbol{z}^h - \boldsymbol{v}_h\Vert_{\textbf{curl},\Omega}.
\end{align*}
Then the estimate for $\boldsymbol{z} - \boldsymbol{z}_h$ is obtained using the triangular inequality.
\end{proof}

Let $\Lambda = \{h_n\}_{n=1}^{\infty}$ be such that $h_n \rightarrow 0$ as $n\rightarrow \infty$. 
Unlike the compact embedding of $\boldsymbol{\mathcal{Z}}(\Omega)$ into $\boldsymbol{L}^2(\Omega)$, 
$\boldsymbol{\mathcal{Z}}_h$ is not a subset of $\boldsymbol{\mathcal{Z}}(\Omega)$. 
Thus $\boldsymbol{\mathcal{Z}}_h$ does not have the same compactness property. Yet what holds for $\boldsymbol{\mathcal{Z}}_h$ is the so-called discrete compactness. 
\begin{definition}
We say that $\{\boldsymbol{\mathcal{Z}}_h\}_{h\in \Lambda}$ has the discrete compactness property 
if for each $\{\boldsymbol{v}_h\}_{h\in \Lambda}$ such that $\boldsymbol{v}_h\in \boldsymbol{\mathcal{Z}}_h$ 
and $\Vert\boldsymbol{v}_h\Vert_{\normalfont{\textbf{curl}},\Omega}\leqslant C$ for all $h\in \Lambda$, there exists $\boldsymbol{v}\in \boldsymbol{\mathcal{Z}}$ and a subsequence, 
still denoted as $\{\boldsymbol{v}_h\}$, such that $\Vert\boldsymbol{v}_h - \boldsymbol{v}\Vert_{0,\Omega}\rightarrow 0$ as $h\rightarrow 0$ in $\Lambda$. 
\end{definition}
In the following, we give a proof for the discrete compactness of $\{\boldsymbol{\mathcal{Z}}_h\}_{h\in \Lambda}$.
\begin{lemma}
\label{L:Zh}
The collection of spaces $\{\boldsymbol{\mathcal{Z}}_h\}_{h\in \Lambda}$ has the discrete compactness property.
\end{lemma}
\begin{proof}
Let $\{\boldsymbol{v}_n\}_{n=1}^{\infty}$ be such that $\boldsymbol{v}_n \in \boldsymbol{\mathcal{Z}}_{h_n} \subset V_{h_n}$ 
and $\Vert \boldsymbol{v}_n\Vert_{\textbf{curl},\Omega} \leqslant C$ for all $n$. By definition,
\begin{align}
(\epsilon_r \boldsymbol{v}_n,\nabla q_n) = 0,\qquad \forall q_n\in U_{h_n}.
\label{E:vn}
\end{align}
Let $\boldsymbol{v}_n = \boldsymbol{v}_{n,0} + \nabla p^n$ be the decomposition of $\boldsymbol{v}_n$ by Lemma \ref{L:decomp}, 
i.e., $p^n\in H^1(\Omega)/\mathbb{C}$ is such that
\begin{align}
(\epsilon_r \boldsymbol{v}_n,\nabla q) = (\epsilon_r \nabla p^n,\nabla q),\qquad \forall q\in H^1(\Omega)/\mathbb{C}.
\label{E:pn}
\end{align}
Taking $q = p^n$ in \eqref{E:pn}, we have that $\Vert \nabla p^n\Vert_{0,\Omega} \leqslant C\Vert \boldsymbol{v}_n\Vert_{0,\Omega} \leqslant C$. 
Therefore, $\{\boldsymbol{v}_{n,0}\} \subset \boldsymbol{\mathcal{Z}}(\Omega)$ with $\Vert \boldsymbol{v}_{n,0}\Vert_{\textbf{curl},\Omega}\leqslant C$ for all $n$. 
Due to the compact embedding of $\boldsymbol{\mathcal{Z}}(\Omega)$ into $\boldsymbol{L}^2(\Omega)$, there exists $\boldsymbol{v}\in \boldsymbol{L}^2(\Omega)$ 
and a subsequence of $\{\boldsymbol{v}_{n,0}\}$, still denoted by $\{\boldsymbol{v}_{n,0}\}$, such that
\begin{align*}
\Vert \boldsymbol{v}_{n,0} - \boldsymbol{v} \Vert_{0,\Omega} \longrightarrow 0,\qquad \text{as}\; n\rightarrow \infty.
\end{align*}
Since $\boldsymbol{v}$ coincides with the weak limit of $\boldsymbol{v}_{n,0}$, $\boldsymbol{v}$ belongs to $\boldsymbol{\mathcal{Z}}(\Omega)$.

On the other hand, by Lemma \ref{L:regul}, $\boldsymbol{v}_{n,0} \in \boldsymbol{H}^{1/2+s}(\Omega)$. Furthermore, due to the de Rham complex, 
$\textbf{curl}\,\boldsymbol{v}_{n,0} = \textbf{curl}\,\boldsymbol{v}_n \in W_{h_n}$. Consequently, (2.4) of \cite{Hsiao2002} is applicable, 
which guarantees that the interpolation operator $\pi_{h_n}^1$ is well-defined for $\boldsymbol{v}_{n,0}$, and for each element $K\in \tau_h$, it holds that
\begin{align}
\Vert (I-\pi_{h_n}^1)\boldsymbol{v}_{n,0}\Vert_{0,K}
\leqslant C\Big(h_n^{1/2+s}\Vert \boldsymbol{v}_{n,0}\Vert_{1/2+s,K} + h_n\Vert \textbf{curl}\,\boldsymbol{v}_{n,0}\Vert_{0,K}\Big).
\label{IE:I-pi}
\end{align}
For this reason, $\pi_{h_n}^1\nabla p^n$ is also well-defined. By using \eqref{E:de Rham}, we have that
\begin{align*}
\textbf{curl}\,\pi_{h_n}^1\nabla p^n = \pi_{h_n}^2 \textbf{curl}\,\nabla p^n = \boldsymbol{0}.
\end{align*}
Since $\pi_{h_n}^1\nabla p^n$ belongs to $V_{h_n}$, there exists $\phi_n\in U_{h_n}$ such that $\pi_{h_n}^1\nabla p^n = \nabla \phi_n$. Combining \eqref{E:vn} and \eqref{E:pn}, we obtain that
\begin{align*}
(\epsilon_r\nabla p^n,\nabla p^n) = (\epsilon_r\nabla p^n,(I-\pi_{h_n}^1)\nabla p^n) = -(\epsilon_r\nabla p^n,(I-\pi_{h_n}^1)\boldsymbol{v}_{n,0}),
\end{align*}
which implies that
\begin{align*}
c\Vert \nabla p^n\Vert_{0,\Omega} 
&\leqslant \Vert (I-\pi_{h_n}^1)\boldsymbol{v}_{n,0}\Vert_{0,\Omega} = \Big(\sum_K \Vert (I-\pi_{h_n}^1)\boldsymbol{v}_{n,0}\Vert^2_{0,K}\Big)^{1/2}.
\end{align*}
Using the above inequality, \eqref{IE:I-pi} and Lemma \ref{L:regul}, we have that
\begin{align*}
\Vert \nabla p^n\Vert_{0,\Omega} &\leqslant C\Big(h_n^{1/2+s}\Vert\boldsymbol{v}_{n,0}\Vert_{1/2+s,\Omega} + h_n\Vert \textbf{curl}\,\boldsymbol{v}_{n,0}\Vert_{0,\Omega}\Big)
\\
&\leqslant C\Big(h_n^{1/2+s}\Vert\boldsymbol{v}_{n,0}\Vert_{\textbf{curl},\Omega} + h_n\Vert \boldsymbol{v}_{n,0}\Vert_{\textbf{curl},\Omega}\Big)
\\
&\leqslant C h_n^{1/2+s} \longrightarrow 0,
\end{align*}
as $n \rightarrow \infty$.
Therefore,
\begin{align*}
\Vert \boldsymbol{v}_n - \boldsymbol{v}\Vert_{0,\Omega} \leqslant \Vert \boldsymbol{v}_{n,0} - \boldsymbol{v}\Vert_{0,\Omega} + \Vert \nabla p^n\Vert_{0,\Omega} \longrightarrow 0,\qquad \text{as}\; n\rightarrow \infty.
\end{align*}
The proof is complete. 
\end{proof}

\begin{definition}
If for each bounded set $\mathcal{A}\subset \boldsymbol{L}^2(\Omega)$, 
$\{K_h\boldsymbol{v}\,\big|\,\forall \boldsymbol{v}\in \mathcal{A},\;\forall h\in \Lambda\}$ is relatively compact in $\boldsymbol{L}^2(\Omega)$, 
we say that the set of bounded linear operators $\{K_h\}_{h\in\Lambda}$ is collectively compact from $\boldsymbol{L}^2(\Omega)$ to $\boldsymbol{L}^2(\Omega)$.
\end{definition} 
By the same arguments for Theorem 7.18 of \cite{Monk2003}, we can deduce from the discrete compactness of $\{\boldsymbol{\mathcal{Z}}_h\}_{h\in \Lambda}$ the collectively compactness of $\{K_h\}_{h\in \Lambda}$.
\begin{lemma}
\label{L:Kh}
$\{K_h\}_{h\in \Lambda}$ is collectively compact from $\boldsymbol{L}^2(\Omega)$ to $\boldsymbol{L}^2(\Omega)$.
\end{lemma}

Now we are in the position to prove the well-posedness of \eqref{E:uh_Kh} and the error estimate for $\boldsymbol{u}_h$. 
Notice that when $\boldsymbol{g}$ in \eqref{E:Kg1}-\eqref{E:Kg2} belongs to $\boldsymbol{\mathcal{Z}}(\Omega)$ and 
$\boldsymbol{g}$ in \eqref{E:Kh g1}-\eqref{E:Kh g2} belongs to $\boldsymbol{\mathcal{Z}}_h$, 
by letting $\boldsymbol{v} = \nabla \phi$ and $\boldsymbol{v}_h = \nabla \phi_h$, 
we see that $\phi = 0$ in $H^1(\Omega)/\mathbb{C}$ and $\phi_h = 0$ in $U_h/\mathbb{C}$.
\begin{theorem}
\label{T:uh}
There exists a unique solution $\boldsymbol{u}_h\in \boldsymbol{\mathcal{Z}}_h$ of \eqref{E:uh_Kh}. Furthermore,
\begin{align*}
\Vert \boldsymbol{u} - \boldsymbol{u}_h\Vert_{\normalfont{\textbf{curl}},\Omega} \leqslant C h^{1/2}\Vert \boldsymbol{f}\Vert_{0,\Gamma} + C\Vert \boldsymbol{f} - \boldsymbol{f}_h\Vert_{0,\Gamma}.
\end{align*}
\end{theorem}
\begin{proof}
By Lemma \ref{L:Kh} and Theorem 2.51 of \cite{Monk2003}, there exists a unique solution of \eqref{E:uh} with
\begin{align*}
\Vert \boldsymbol{u}_h\Vert_{0,\Omega} \leqslant C\Vert \boldsymbol{z}_h\Vert_{0,\Omega}\leqslant C\Vert \boldsymbol{f}_h\Vert_{0,\Gamma}
\end{align*}
and
\begin{align*}
\Vert \boldsymbol{u} - \boldsymbol{u}_h\Vert_{0,\Omega}\leqslant C\Big(\Vert \boldsymbol{z} - \boldsymbol{z}_h\Vert_{0,\Omega} + \Vert(K - K_h)\boldsymbol{u}\Vert_{0,\Omega}\Big).
\end{align*}
By \eqref{E:uh_Kh}, it holds that
\begin{align*}
\Vert \boldsymbol{u}_h\Vert_{\textbf{curl},\Omega} = \Vert \boldsymbol{z}_h - K_h\boldsymbol{u}_h\Vert_{\textbf{curl},\Omega} \leqslant C\Vert\boldsymbol{f}_h\Vert_{0,\Gamma} + C\Vert \boldsymbol{u}_h\Vert_{0,\Omega} \leqslant C\Vert \boldsymbol{f}_h\Vert_{0,\Gamma}.
\end{align*}
Meanwhile, using \eqref{E:u_K} and \eqref{E:uh_Kh}, we have that
\begin{align*}
\Vert \boldsymbol{u} - \boldsymbol{u}_h\Vert_{\textbf{curl},\Omega} &= \Vert \boldsymbol{z} - K\boldsymbol{u} - \boldsymbol{z}_h + K_h\boldsymbol{u}_h\Vert_{\textbf{curl},\Omega}
\\
&\leqslant \Vert\boldsymbol{z} - \boldsymbol{z}_h\Vert_{\textbf{curl},\Omega} + \Vert(K - K_h)\boldsymbol{u}\Vert_{\textbf{curl},\Omega} + \Vert K_h(\boldsymbol{u}-\boldsymbol{u}_h)\Vert_{\textbf{curl},\Omega}.
\end{align*}
Due to the well-posedness of \eqref{E:Kh g1}-\eqref{E:Kh g2}, 
there exists a constant $C$ independent of $h$ such that $\Vert K_h(\boldsymbol{u} - \boldsymbol{u}_h)\Vert_{\textbf{curl},\Omega} \leqslant C \Vert \boldsymbol{u} - \boldsymbol{u}_h\Vert_{0,\Omega}$.
Combining the above results, we get
\begin{align*}
\Vert \boldsymbol{u} - \boldsymbol{u}_h\Vert_{\normalfont{\textbf{curl}},\Omega} \leqslant C\Big(\Vert \boldsymbol{z} - \boldsymbol{z}_h\Vert_{\normalfont{\textbf{curl}},\Omega} + \Vert (K - K_h)\boldsymbol{u}\Vert_{\normalfont{\textbf{curl}},\Omega}\Big).
\end{align*}
Together with Lemma \ref{L:zh}, Lemma 5.1 of \cite{Bermudez2005} and the regularity results given in Lemma \ref{L:regul_Ku,z}, we obtain the desired estimate.
\end{proof}

\section{The Eigenvalue Problem and its FE Approximation}
The modified Maxwell's Stekloff eigenvalue problem is to find $\lambda\in \mathbb{C}$ and non-trivial $\boldsymbol{u}\in \boldsymbol{H}(\textbf{curl};\Omega)$ such that
\begin{align}
\label{E:eig}
a(\boldsymbol{u},\boldsymbol{v}) = -\lambda\langle\boldsymbol{\mathcal{S}}\boldsymbol{u}_T,\boldsymbol{v}_T\rangle,\qquad \forall \boldsymbol{v}\in \boldsymbol{H}(\textbf{curl};\Omega).
\end{align}
Here $\boldsymbol{\mathcal{S}}$ is defined by
\begin{align*}
\boldsymbol{\mathcal{S}}\; : \; \boldsymbol{L}^2_t(\Gamma)\; \longrightarrow \; \boldsymbol{H}(\text{div}^0_{\Gamma};\Gamma),\qquad \boldsymbol{\mu}\; \longmapsto \; \textbf{curl}_{\Gamma}\, q,
\end{align*}
where $q \in H^1(\Gamma)/\mathbb{C}$ is the solution of the problem
\begin{align}
\langle \textbf{curl}_{\Gamma}\, q,\textbf{curl}_{\Gamma}\, \psi\rangle = \langle \boldsymbol{\mu},\textbf{curl}_{\Gamma}\,\psi\rangle,\qquad \forall \psi\in H^1(\Gamma)/\mathbb{C}.
\label{E:q}
\end{align}
Let $\boldsymbol{u}$ be the solution of \eqref{E:u} with $\boldsymbol{f}\in \boldsymbol{H}(\text{div}_{\Gamma}^0;\Gamma)$. We define the Neumann-to-Dirichlet operator $\boldsymbol{T}$ of \eqref{E:u} by
\begin{align*}
\boldsymbol{T}\; : \;\boldsymbol{H}(\text{div}_{\Gamma}^0;\Gamma)\;\longrightarrow \;\boldsymbol{H}(\text{div}_{\Gamma}^0;\Gamma),\qquad \boldsymbol{f}\;\longmapsto\;\boldsymbol{\mathcal{S}}\boldsymbol{u}_T.
\end{align*}
Using the similar arguments as in \cite{Camano2017}, we can show that $\boldsymbol{T}$ is compact and self-adjoint for Lipschitz polyhedra.

\begin{lemma}
\label{L:T}
The operator $\boldsymbol{T} : \boldsymbol{H}(\normalfont{\text{div}}_{\Gamma}^0;\Gamma) \rightarrow \boldsymbol{H}(\normalfont{\text{div}}_{\Gamma}^0;\Gamma)$ is compact and self-adjoint. 
\end{lemma}
\begin{proof}
Given $\boldsymbol{f}\in \boldsymbol{H}(\text{div}_{\Gamma}^0;\Gamma)$, 
let $\boldsymbol{u}$ be the solution of \eqref{E:u} and $q\in H^1(\Gamma)/\mathbb{C}$ the solution of \eqref{E:q} with $\boldsymbol{\mu} = \boldsymbol{u}_T$. 
By the regularity result in \cite{Costabel1990}, $\text{curl}_{\Gamma}\,\boldsymbol{u}_T \in \boldsymbol{L}^2_t(\Gamma)$ and
\begin{align*}
\Vert\text{curl}_{\Gamma}\,\boldsymbol{u}_T\Vert_{0,\Gamma} = \Vert\boldsymbol{\nu}\cdot\textbf{curl}\,\boldsymbol{u}\Vert_{0,\Gamma} &\leqslant C(\Vert\textbf{curl}\,\boldsymbol{u}\Vert_{0,\Omega} + \Vert\textbf{curl}\,\textbf{curl}\,\boldsymbol{u}\Vert_{0,\Omega} + \Vert\boldsymbol{\nu}\times\textbf{curl}\,\boldsymbol{u}\Vert_{0,\Gamma})
\\
&\leqslant C\Vert \boldsymbol{f}\Vert_{0,\Gamma}.
\end{align*}
We apply Theorem 8 of \cite{Buffa2002} to claim that $q$ belongs to $H^{1+t}(\Gamma)/\mathbb{C}$ for $0 \leqslant t < s_3$ where
\begin{equation}\label{s3}
s_3 := \min\{s_{\Gamma},1\}
\end{equation} 
with $s_{\Gamma} > 0$ depending on the geometry of $\Gamma$.
In addition, (2.2) of \cite{Hiptmair2002} shows that
\begin{align}
\Vert q\Vert_{H^{1+t}(\Gamma)/\mathbb{C}}\leqslant C\Vert \text{curl}_{\Gamma}\,\boldsymbol{u}_T\Vert_{0,\Gamma} \leqslant C\Vert \boldsymbol{f}\Vert_{0,\Gamma}.
\label{IE:q}
\end{align}
Hence $\boldsymbol{T}\boldsymbol{f} = \boldsymbol{\mathcal{S}}\boldsymbol{u}_T = \textbf{curl}_{\Gamma}\, q \in \boldsymbol{H}^{t}(\text{div}^0_{\Gamma},\Gamma)$, which implies that $\boldsymbol{T}$ is compact.

Given $\boldsymbol{f}$, $\boldsymbol{g}\in \boldsymbol{H}(\text{div}_{\Gamma}^0;\Gamma)$, to see that $\boldsymbol{T}$ is self-adjoint, 
let $\boldsymbol{u}$ and $\boldsymbol{v}$ be the solutions of \eqref{E:u} with data $\boldsymbol{f}$ and $\boldsymbol{g}$, respectively. 
Then 
\[
\langle \boldsymbol{T}\boldsymbol{f},\boldsymbol{g}\rangle = \langle \boldsymbol{\mathcal{S}}\boldsymbol{u}_T,\boldsymbol{g}\rangle = \langle \boldsymbol{u}_T,\boldsymbol{g}\rangle = \overline{\langle\boldsymbol{g},\boldsymbol{u}_T\rangle} = \overline{a(\boldsymbol{v},\boldsymbol{u})} = a(\boldsymbol{u},\boldsymbol{v}) = \langle\boldsymbol{f},\boldsymbol{v}_T\rangle = \langle\boldsymbol{f},\boldsymbol{\mathcal{S}}\boldsymbol{v}_T\rangle = \langle \boldsymbol{f},\boldsymbol{T}\boldsymbol{g}\rangle.
\] 
The proof is complete.
\end{proof}

There is a one-to-one correspondence between the eigenpairs of \eqref{E:eig} and those of $\boldsymbol{T}$. 
In fact, if $(\lambda,\boldsymbol{u})$ is an eigenpair of \eqref{E:eig}, 
then $\boldsymbol{T}(-\lambda\boldsymbol{\mathcal{S}}\boldsymbol{u}_T) = \boldsymbol{\mathcal{S}}\boldsymbol{u}_T$.
Hence $(-1/\lambda,\boldsymbol{\mathcal{S}}\boldsymbol{u}_T)$ is an eigenpair of $\boldsymbol{T}$. 
On the other hand, if $(\mu,\boldsymbol{g})$ is an eigenpair of $\boldsymbol{T}$, then letting $\boldsymbol{w}$ be the solution of \eqref{E:u} with data $\boldsymbol{g}$, 
we see that $\boldsymbol{\mathcal{S}}\boldsymbol{w}_T = \boldsymbol{T}\boldsymbol{g} = \mu\boldsymbol{g}$. Thus $(-1/\mu,\boldsymbol{w})$ is an eigenpair of \eqref{E:eig}.
Then the existence of a discrete set of eigenvalues of \eqref{E:eig} is guaranteed by Lemma \ref{L:T}.

In the following we propose a finite element method for \eqref{E:eig}.
To approximate the operator $\boldsymbol{\mathcal{S}}$, an equivalent form of $\boldsymbol{\mathcal{S}}$ is considered in \cite{Camano2017}:
\begin{align*}
\boldsymbol{\mathcal{S}}\boldsymbol{\mu} = \boldsymbol{\mu} + \nabla_{\Gamma}\, p,
\end{align*}
where $p\in H^1(\Gamma)/\mathbb{C}$ is the solution of 
\begin{align*}
\langle \nabla_{\Gamma}\, p,\nabla_{\Gamma}\,\psi\rangle = -\langle \boldsymbol{\mu},\nabla_{\Gamma}\, \psi\rangle,\qquad \forall \psi\in H^1(\Gamma)/\mathbb{C}.
\end{align*}
The approximation of $\boldsymbol{\mathcal{S}}$ in \cite{Camano2017} is defined as
\begin{align*}
\boldsymbol{\mathcal{S}}^+_h\boldsymbol{\mu}_h = \boldsymbol{\mu}_h + \nabla_{\Gamma}\, p_h
\end{align*}
with $p_h\in \partial U_h/\mathbb{C}$ being the solution of
\begin{align*}
\langle \nabla_{\Gamma}\, p_h,\nabla_{\Gamma}\, \psi_h\rangle = -\langle \boldsymbol{\mu}_h,\nabla_{\Gamma}\,\psi_h\rangle,\qquad \forall \psi_h\in \partial U_h/\mathbb{C}.
\end{align*}
Here $\partial U_h$ represents the finite element space of $H^1(\Gamma)$. 
Unfortunately, the range of $\boldsymbol{\mathcal{S}}^+_h$ is no longer a subset of $\boldsymbol{H}(\text{div}_{\Gamma}^0;\Gamma)$,
which complicates the analysis. Nonetheless, the numerical results show that the use of $\boldsymbol{\mathcal{S}}^+_h$ computes correct eigenvalues.

To this end, we define a different discrete operator $\boldsymbol{\mathcal{S}}_h$ based on the original expression of $\boldsymbol{\mathcal{S}}$
(see also \cite{Halla2019approximation}):
\begin{align}
\boldsymbol{\mathcal{S}}_h\; : \; \boldsymbol{L}^2_t(\Gamma)\; \longrightarrow \; \boldsymbol{H}(\text{div}^0_{\Gamma};\Gamma),\quad \boldsymbol{\mu}_h\; \longmapsto \; \textbf{curl}_{\Gamma}\, q_h,
\label{E:Sh}
\end{align}
where $q_h\in \partial U_h/\mathbb{C}$ is the solution of
\begin{align}
\langle \textbf{curl}_{\Gamma}\, q_h,\textbf{curl}_{\Gamma}\, \psi_h\rangle = \langle\boldsymbol{\mu}_h,\textbf{curl}_{\Gamma}\, \psi_h\rangle,\qquad \forall \psi_h\in \partial U_h/\mathbb{C}.
\label{E:qh}
\end{align}
The discrete Stekloff eigenvalue problem is then to find $(\lambda_h,\boldsymbol{u}_h)\in \mathbb{C}\times V_h$ such that
\begin{align}
a(\boldsymbol{u}_h,\boldsymbol{v}_h) = -\lambda_h\langle \boldsymbol{\mathcal{S}}_h\boldsymbol{u}_{h,T},\boldsymbol{v}_{h,T}\rangle,\qquad \forall \boldsymbol{v}_h\in V_h,
\label{E:eig_h}
\end{align}
where $\boldsymbol{v}_{h,T} := (\boldsymbol{v}_h)_T$.
Let $\boldsymbol{u}_h$ be the solution of \eqref{E:uh} given $\boldsymbol{f}_h = \boldsymbol{f} \in \boldsymbol{H}(\text{div}_{\Gamma}^0;\Gamma)$. 
The corresponding discrete Neumann-to-Dirichlet operator $\boldsymbol{T}_h$ can be defined as
\begin{align*}
\boldsymbol{T}_h\; : \;\boldsymbol{H}(\text{div}^0_{\Gamma};\Gamma)\; \longrightarrow \;\boldsymbol{H}(\text{div}^0_{\Gamma};\Gamma),\qquad  \boldsymbol{f}\;\longmapsto \;\boldsymbol{\mathcal{S}}_h\boldsymbol{u}_{h,T}.
\end{align*}
It is easy to check that if $(\lambda_h,\boldsymbol{u}_h)$ is an eigenpair of \eqref{E:eig_h}, then $(-1/\lambda_h,\boldsymbol{\mathcal{S}}_h\boldsymbol{u}_{h,T})$ is an eigenpair of $\boldsymbol{T}_h$; if $(\mu_h,\boldsymbol{g}_h)$ is an eigenpair of $\boldsymbol{T}_h$, then $(-1/\mu_h,\boldsymbol{v}_h)$ is an eigenpair of \eqref{E:eig_h}, where $\boldsymbol{v}_h$ is the solution of \eqref{E:uh} with data $\boldsymbol{g}_h$.
Note that $\boldsymbol{T}_h$ is not self-adjoint.
Its adjoint operator is $\boldsymbol{T}_h^* : \boldsymbol{H}(\text{div}^0_{\Gamma};\Gamma)\; \rightarrow \;\boldsymbol{H}(\text{div}^0_{\Gamma};\Gamma)$, $\boldsymbol{g}\mapsto\boldsymbol{\mathcal{S}}\boldsymbol{v}_{h,T}$ with $\boldsymbol{v}_h$ the solution of \eqref{E:uh} with data $\boldsymbol{\mathcal{S}}_h\boldsymbol{g}$. 
In fact, given $\boldsymbol{f}, \boldsymbol{g}\in \boldsymbol{H}(\text{div}^0_{\Gamma};\Gamma)$, 
we let $\boldsymbol{u}_h$ and $\boldsymbol{v}_h$ be, respectively, the solutions of \eqref{E:uh} with data $\boldsymbol{f}$ and $\boldsymbol{\mathcal{S}}_h \boldsymbol{g}$. 
Then $\langle \boldsymbol{T}_h\boldsymbol{f},\boldsymbol{g}\rangle = \langle\boldsymbol{\mathcal{S}}_h\boldsymbol{u}_{h,T},\boldsymbol{g}\rangle = \langle \boldsymbol{\mathcal{S}}_h\boldsymbol{u}_{h,T},\boldsymbol{\mathcal{S}}_h\boldsymbol{g}\rangle = \langle\boldsymbol{u}_{h,T},\boldsymbol{\mathcal{S}}_h\boldsymbol{g}\rangle = \overline{\langle\boldsymbol{\mathcal{S}}_h\boldsymbol{g},\boldsymbol{u}_{h,T}\rangle} = \overline{a(\boldsymbol{v}_h,\boldsymbol{u}_h)} = a(\boldsymbol{u}_h,\boldsymbol{v}_h) = \langle \boldsymbol{f},\boldsymbol{v}_{h,T}\rangle = \langle\boldsymbol{f},\boldsymbol{\mathcal{S}}\boldsymbol{v}_{h,T}\rangle$.

To estimate $\boldsymbol{T}-\boldsymbol{T}_h$, 
we split the error $(\boldsymbol{T} - \boldsymbol{T}_h)\boldsymbol{f}$ into $(\boldsymbol{\mathcal{S}} - \boldsymbol{\mathcal{S}}_h)\boldsymbol{u}_T$ 
and $\boldsymbol{\mathcal{S}}_h(\boldsymbol{u}_T - \boldsymbol{u}_{h,T})$, and treat them separately.

\begin{lemma}
\label{L:T1}
Let $\boldsymbol{u}\in \boldsymbol{\mathcal{Z}}(\Omega)$ be the solution of \eqref{E:u} with $\boldsymbol{f}\in \boldsymbol{H}(\normalfont{\text{div}}_{\Gamma}^0;\Gamma)$. 
Then, for $0\leqslant t<s_3$, where $s_3$ is given by \eqref{s3}, it holds that
\begin{align*}
\Vert (\boldsymbol{\mathcal{S}} - \boldsymbol{\mathcal{S}}_h)\boldsymbol{u}_T\Vert_{0,\Gamma} \leqslant Ch^t \Vert \boldsymbol{f}\Vert_{0,\Gamma}.
\end{align*}
\end{lemma}
\begin{proof}
Let $q\in H^1(\Gamma)/\mathbb{C}$ and $\widehat{q}_h\in \partial U_h/\mathbb{C}$ be the solutions of \eqref{E:q} and \eqref{E:qh} with 
$\boldsymbol{\mu} = \boldsymbol{u}_T$ and $\boldsymbol{\mu}_h = \boldsymbol{u}_T$, respectively.
By Cea's lemma,
\begin{align*}
\Vert \textbf{curl}_{\Gamma}\,q - \textbf{curl}_{\Gamma}\, \widehat{q}_h\Vert_{0,\Gamma} &\leqslant \Vert\textbf{curl}_{\Gamma}\,q - \textbf{curl}_{\Gamma}\,\pi^{\Gamma}_h q\Vert_{0,\Gamma}
\leqslant \Vert q - \pi_h^{\Gamma}q\Vert_{H^1(\Gamma)/\mathbb{C}}\leqslant C h^t\Vert q\Vert_{H^{1+t}(\Gamma)/\mathbb{C}},
\end{align*}
where $\pi^{\Gamma}_h$ stands for the interpolation from $H^1(\Gamma)$ to $\partial U_h$. Using \eqref{IE:q}, we have that
\[
\Vert\textbf{curl}_{\Gamma}\, q - \textbf{curl}_{\Gamma}\, \widehat{q}_h\Vert_{0,\Gamma} \leqslant C h^t \Vert\boldsymbol{f}\Vert_{0,\Gamma},
\]
which is the desired inequality.
\end{proof}

\begin{lemma}
\label{L:T2}
Let $\boldsymbol{u}\in \boldsymbol{\mathcal{Z}}(\Omega)$ and $\boldsymbol{u}_h\in \boldsymbol{\mathcal{Z}}_h$ 
be, respectively, the solutions of \eqref{E:u} and \eqref{E:uh} with the same data $\boldsymbol{f}\in \boldsymbol{H}(\normalfont{\text{div}}_{\Gamma}^0;\Gamma)$. 
Then
\begin{align*}
\Vert \boldsymbol{\mathcal{S}}_h(\boldsymbol{u}_T - \boldsymbol{u}_{h,T})\Vert_{0,\Gamma}
\leqslant C h^{1/2} \Vert \boldsymbol{f}\Vert_{0,\Gamma}.
\end{align*}
\end{lemma}

\begin{proof}
Let $\boldsymbol{u}_h = \boldsymbol{u}_{h,0} + \nabla p^h$ be the decomposition of $\boldsymbol{u}_h$ due to Lemma \ref{L:decomp}. 
We estimate the norm of $\nabla p^h$ following the same procedure as Lemma \ref{L:Zh} and use Theorem \ref{T:uh} to deduce that
\begin{align*}
\Vert \boldsymbol{u} - \boldsymbol{u}_{h,0}\Vert_{\textbf{curl},\Omega} &\leqslant \Vert \nabla p^h\Vert_{0,\Omega} + \Vert \boldsymbol{u} - \boldsymbol{u}_h\Vert_{\textbf{curl},\Omega} \leqslant C h^{1/2+s}\Vert \boldsymbol{u}_{h,0}\Vert_{\textbf{curl},\Omega} + C h^{1/2}\Vert \boldsymbol{f}\Vert_{0,\Gamma}
\\
&\leqslant Ch^{1/2}\Big(\Vert \boldsymbol{u}_h\Vert_{\textbf{curl},\Omega} + \Vert\boldsymbol{f}\Vert_{0,\Gamma} \Big) \leqslant Ch^{1/2} \Vert\boldsymbol{f}\Vert_{0,\Gamma}.
\end{align*}
By the definition of $\boldsymbol{\mathcal{S}}_h$, we have 
\begin{align*}
\Vert\boldsymbol{\mathcal{S}}_h(\boldsymbol{u}_T - \boldsymbol{u}_{h,T})\Vert_{0,\Gamma} = \Vert\boldsymbol{\mathcal{S}}_h(\boldsymbol{u}_T - \boldsymbol{u}_{h,0,T})\Vert_{0,\Gamma} 
&\leqslant \Vert \boldsymbol{u}_T - \boldsymbol{u}_{h,0,T}\Vert_{0,\Gamma} 
\\
&\leqslant C\Vert \boldsymbol{u} - \boldsymbol{u}_{h,0}\Vert_{\textbf{curl},\Omega} \leqslant C h^{1/2}\Vert \boldsymbol{f}\Vert_{0,\Gamma},
\end{align*}
where we have used Lemma \ref{L:regul} for $\boldsymbol{u} - \boldsymbol{u}_{h,0}$.
\end{proof}

Combining Lemmas \ref{L:T1} and  \ref{L:T2}, we obtain the convergence of $\boldsymbol{T}_h$ to $\boldsymbol{T}$.
\begin{theorem}
\label{T:T}
For $0\leqslant t < s_3$, we have
\begin{align*}
\Vert \boldsymbol{T} - \boldsymbol{T}_h\Vert \leqslant Ch^{\min\{1/2,t\}}.
\end{align*}
\end{theorem}

We have shown that $\boldsymbol{T}$ and $\boldsymbol{T}_h$ are compact and $\boldsymbol{T}$ is self-adjoint. 
In addition, $\boldsymbol{T}_h$ converges to $\boldsymbol{T}$ in norm. 
In the following, we apply the Babu\v{s}ka-Osborn theory \cite{Babuska1991} to show the convergence order of the eigenvalues of $\boldsymbol{T}_h$. 
Let $\mu$ be a non-zero
eigenvalue of $\boldsymbol{T}$ with multiplicity $m$ and $\mu_{j,h}, j = 1, \ldots, m,$ be the eigenvalues of $\boldsymbol{T}_h$ that approximate $\mu$. 
For a simple closed curve $\Gamma \subset \rho(\boldsymbol{T})$ which encloses only one eigenvalue $\mu$ of $\boldsymbol{T}$, we denote the projection operator $E(\mu)$ by
\begin{align*}
E(\mu) = \frac{1}{2\pi i}\int_{\Gamma}(z - \boldsymbol{T})^{-1} dz.
\end{align*}
Let $\boldsymbol{f}_1,\dots,\boldsymbol{f}_m\in E := E(\mu)\boldsymbol{H}(\text{div}_{\Gamma}^0;\Gamma) \subset \boldsymbol{H}(\text{div}_{\Gamma}^0;\Gamma)$ 
be a basis of eigenvectors of $\mu$ with $\Vert\boldsymbol{f}_i\Vert_{0,\Gamma} = 1$ for $i = 1,\dots, m$. 
Since $\boldsymbol{T}$ is self-adjoint, $E$ is the span of $\boldsymbol{f}_1,\dots,\boldsymbol{f}_m$.
Define
\begin{equation}\label{hatmuh}
\widehat{\mu}_h = \frac{1}{m}\sum_{j=1}^m \mu_{j,h}.
\end{equation}
The following theorem shows that $\widehat{\mu}_h$ converges to $\mu$ with order at least $\min\{1,2t\}$.
\begin{theorem}
\label{T:mu}
Let $\mu$ be an eigenvalue of $\boldsymbol{T}$ and $\widehat{\mu}_h$ be defined by \eqref{hatmuh}. It holds that, for $0\leqslant t < s_3$,
\begin{align*}
|\mu - \widehat{\mu}_h| \leqslant C h^{\min\{1,2t\}}.
\end{align*}
\end{theorem}
\begin{proof}
By the Babu\v{s}ka-Osborn theory \cite{Babuska1991}, 
\begin{align}
|\mu - \widehat{\mu}_h|\leqslant \frac{1}{m}\sum_{j=1}^m |\langle (\boldsymbol{T} - \boldsymbol{T}_h)\boldsymbol{f}_j,\boldsymbol{f}_j\rangle| + C\Vert (\boldsymbol{T} - \boldsymbol{T}_h)|_{E}\Vert \Vert(\boldsymbol{T} - \boldsymbol{T}_h^*)|_{E}\Vert.
\label{E:B-O}
\end{align}
Let $\boldsymbol{u}^j$ and $\boldsymbol{u}^j_h$ be the solutions of \eqref{E:u} and \eqref{E:uh} with $\boldsymbol{f} = \boldsymbol{f}_j$ and
$\boldsymbol{f}_h = \boldsymbol{f}_j$, respectively.
Let $\widetilde{\boldsymbol{u}}^i$ and $\widetilde{\boldsymbol{u}}^i_h$ be the solutions of \eqref{E:u} and \eqref{E:uh} 
with $\boldsymbol{f} = \boldsymbol{\mathcal{S}}_h\boldsymbol{f}_i$ and $\boldsymbol{f}_h = \boldsymbol{\mathcal{S}}_h\boldsymbol{f}_i$,  respectively. We have that
\begin{align}
a(\boldsymbol{u}^j,\boldsymbol{v}) = \langle\boldsymbol{f}_j,\boldsymbol{v}_T\rangle,\qquad\quad\;\, a(\widetilde{\boldsymbol{u}}^i,\boldsymbol{v}) = \langle\boldsymbol{\mathcal{S}}_h\boldsymbol{f}_i,\boldsymbol{v}_T\rangle,\qquad\;\;\, &\forall \boldsymbol{v}\in \boldsymbol{H}(\textbf{curl};\Omega),
\label{E:u^j}
\\
a(\boldsymbol{u}^j_h,\boldsymbol{v}_h) = \langle\boldsymbol{f}_j,\boldsymbol{v}_{h,T}\rangle,\qquad a(\widetilde{\boldsymbol{u}}^i_h,\boldsymbol{v}_h) = \langle\boldsymbol{\mathcal{S}}_h\boldsymbol{f}_i,\boldsymbol{v}_{h,T}\rangle,\qquad &\forall \boldsymbol{v}_h\in V_h,
\label{E:u^j_h}
\end{align}
Using \eqref{E:u^j}-\eqref{E:u^j_h} and Theorem \ref{T:uh}, we obtain that
\begin{align*}
\vert \langle (\boldsymbol{T} - \boldsymbol{T}_h) \boldsymbol{f}_j,\boldsymbol{\mathcal{S}}_h\boldsymbol{f}_i\rangle\vert &= \vert \langle \boldsymbol{\mathcal{S}}\boldsymbol{u}^j_T - \boldsymbol{\mathcal{S}}_h\boldsymbol{u}^j_{h,T},\boldsymbol{\mathcal{S}}_h\boldsymbol{f}_i\rangle\vert = \vert \langle\boldsymbol{u}_T^j - \boldsymbol{u}_{h,T}^j,\boldsymbol{\mathcal{S}}_h\boldsymbol{f}_i\rangle\vert = \vert a(\boldsymbol{u}^j - \boldsymbol{u}_h^j,\widetilde{\boldsymbol{u}}^i)\vert
\\
&= \vert a(\boldsymbol{u}^j - \boldsymbol{u}^j_h,\widetilde{\boldsymbol{u}}^i - \widetilde{\boldsymbol{u}}^i_h)\vert \leqslant \Vert \boldsymbol{u}^j - \boldsymbol{u}^j_h\Vert_{\textbf{curl},\Omega}\Vert \widetilde{\boldsymbol{u}}^i - \widetilde{\boldsymbol{u}}^i_h\Vert_{\textbf{curl},\Omega} 
\\
&\leqslant Ch^{1/2}\Vert \boldsymbol{f}_j\Vert_{0,\Gamma} \, Ch^{1/2}\Vert \boldsymbol{\mathcal{S}}_h\boldsymbol{f}_i\Vert_{0,\Gamma} \leqslant Ch.
\end{align*}
Notice that $\boldsymbol{f}_i = \boldsymbol{\mathcal{S}}\boldsymbol{u}^i_T/\mu$ and $\boldsymbol{\mathcal{S}}_h\boldsymbol{\mathcal{S}} = \boldsymbol{\mathcal{S}}_h$. 
We apply Theorem \ref{T:T} and Lemma \ref{L:T1} to get
\begin{align*}
\vert \langle (\boldsymbol{T} - \boldsymbol{T}_h) \boldsymbol{f}_j,\boldsymbol{f}_i - \boldsymbol{\mathcal{S}}_h\boldsymbol{f}_i\rangle\vert &\leqslant \Vert (\boldsymbol{T} - \boldsymbol{T}_h)\boldsymbol{f}_j\Vert_{0,\Gamma} \Vert (\boldsymbol{\mathcal{S}} - \boldsymbol{\mathcal{S}}_h)\boldsymbol{u}^i_T/\mu\Vert_{0,\Gamma}
\\
&\leqslant C h^{\min\{1/2,t\}}\Vert \boldsymbol{f}_j\Vert_{0,\Gamma} \,C h^t \Vert \boldsymbol{f}_i\Vert_{0,\Gamma} \\
& \leqslant C h^{\min\{1/2,t\}+t}.
\end{align*}
The above two inequalities imply that
\begin{align*}
\frac{1}{m}\sum_{j=1}^m\vert \langle (\boldsymbol{T}-\boldsymbol{T}_h)\boldsymbol{f}_j,\boldsymbol{f}_j\rangle\vert \leqslant C h^{\min\{1,2t\}}.
\end{align*}
On the other hand, $\Vert (\boldsymbol{T} - \boldsymbol{T}_h)|_E\Vert \leqslant \Vert\boldsymbol{T} - \boldsymbol{T}_h\Vert \leqslant Ch^{\min\{1/2,t\}}$ by Theorem \ref{T:T}. 

For $\Vert(\boldsymbol{T} - \boldsymbol{T}_h^*)|_{E}\Vert$, we apply $\boldsymbol{T} - \boldsymbol{T}_h^*$ on the eigenvector $\boldsymbol{f}_i$. 
Decompose $\widetilde{\boldsymbol{u}}^i_h$ into $\widetilde{\boldsymbol{u}}^i_h = \widetilde{\boldsymbol{u}}^i_{h,0} + \nabla p^h$ 
according to Lemma \ref{L:decomp} and use the arguments of Lemma \ref{L:T2} to deduce
\begin{align*}
&\Vert (\boldsymbol{T} - \boldsymbol{T}_h^*)\boldsymbol{f}_i\Vert_{0,\Gamma} = \Vert \boldsymbol{\mathcal{S}}(\boldsymbol{u}^i_T - \widetilde{\boldsymbol{u}}_{h,T}^i)\Vert_{0,\Gamma} = \Vert \boldsymbol{\mathcal{S}}(\boldsymbol{u}^i_T - \widetilde{\boldsymbol{u}}^i_{h,0,T})\Vert_{0,\Gamma} 
\\
\leqslant\;& \Vert \boldsymbol{u}^i_T - \widetilde{\boldsymbol{u}}^i_{h,0,T}\Vert_{0,\Gamma}
\leqslant C\Vert \boldsymbol{u}^i - \widetilde{\boldsymbol{u}}^i_{h,0}\Vert_{\textbf{curl},\Omega} \leqslant C \Vert \nabla p^h\Vert_{0,\Omega} + C\Vert \boldsymbol{u}^i - \widetilde{\boldsymbol{u}}^i_h\Vert_{\textbf{curl},\Omega}
\\
\leqslant\;& Ch^{1/2+s}\Vert\widetilde{\boldsymbol{u}}^i_{h,0}\Vert_{\textbf{curl},\Omega} + Ch^{1/2}\Vert \boldsymbol{f}_i\Vert_{0,\Gamma} + C\Vert \boldsymbol{f}_i - \boldsymbol{\mathcal{S}}_h\boldsymbol{f}_i\Vert_{0,\Gamma}
\\
\leqslant\;& Ch^{1/2}\Vert \boldsymbol{f}_i\Vert_{0,\Gamma} + C\Vert(\boldsymbol{\mathcal{S}} - \boldsymbol{\mathcal{S}}_h)\boldsymbol{u}^i_T/\mu\Vert_{0,\Gamma} \leqslant C h^{\min\{1/2,t\}}\Vert \boldsymbol{f}_i\Vert_{0,\Gamma}.
\end{align*}
Consequently, $\Vert (\boldsymbol{T} - \boldsymbol{T}_h^*)|_E\Vert \leqslant Ch^{\min\{1/2,t\}}$.
The $\min\{1,2t\}$ order convergence is obtained by plugging all the estimates into \eqref{E:B-O}.
\end{proof}

Define $\lambda = -\mu^{-1}$, $\lambda_{j,h} = -\mu_{j,h}^{-1}$, $j = 1,\dots,m,$ and 
\begin{align}
\widehat{\lambda}_h = \frac{1}{m}\sum_{j=1}^m \lambda_{j,h}.
\label{E:lambda_h}
\end{align}
We conclude this section with the estimate of $\lambda - \widehat{\lambda}_h$ using Remark 7.3 of \cite{Babuska1991}.

\begin{corollary}
\label{T:lambda}
Given an eigenvalue $\lambda$ of \eqref{E:eig} and $\widehat{\lambda}_h$ be defined by \eqref{E:lambda_h}, it holds that, for $0\leqslant t < s_3$,
\begin{align*}
|\lambda - \widehat{\lambda}_h| \leqslant C h^{\min\{1,2t\}}.
\end{align*}
\end{corollary}
 

\section{Numerical results}
In this section we present some numerical examples. 
We show the computed eigenvalues of both
\begin{align}
a(\boldsymbol{u}_h,\boldsymbol{v}_h) = -\lambda_h\langle \boldsymbol{\mathcal{S}}_h\boldsymbol{u}_{h,T},\boldsymbol{v}_{h,T}\rangle,\qquad \forall \boldsymbol{v}_h\in V_h,
\label{E:eig_u_h}
\end{align}
and 
\begin{align}
a(\boldsymbol{u}^+_h,\boldsymbol{v}_h) = -\lambda^+_h\langle \boldsymbol{\mathcal{S}}^+_h\boldsymbol{u}^+_{h,T},\boldsymbol{v}_{h,T}\rangle,\qquad \forall \boldsymbol{v}_h\in V_h.
\label{E:eig_u+_h}
\end{align}
Write $\boldsymbol{x}<a$ if $x_1<a, x_2<a$ and $x_3 < a$. 
Consider three different domains including the unit cube $\Omega_1 = \{\boldsymbol{x}\in \mathbb{R}^3\,|\, 0< \boldsymbol{x}< 1\}$, the ``L-shaped'' domain $\Omega_2 = \Omega_1 \cap \{\boldsymbol{x}\in \mathbb{R}^3\,|\, 1/2 \leqslant \boldsymbol{x} \leqslant 1\}^c$ and the unit ball $\Omega_3 = \{\boldsymbol{x}\in \mathbb{R}^3\,|\, |\boldsymbol{x}| < 1\}$. 
The tetrahedral meshes are generated by Gmsh. By the definition of $\boldsymbol{\mathcal{S}}_h$ in \eqref{E:Sh}, the equation \eqref{E:eig_u_h} becomes
\begin{align*}
a(\boldsymbol{u}_h,\boldsymbol{v}_h) &= -\lambda_h\langle \textbf{curl}\, q_h,\boldsymbol{v}_{h,T}\rangle, \qquad\qquad\;\;\, \forall \boldsymbol{v}_h\in V_h,
\\
0 &= \langle \textbf{curl}\, q_h - \boldsymbol{u}_{h,T},\textbf{curl}\, \psi_h\rangle,\qquad \forall \psi_h\in \partial U_h/\mathbb{C}.
\end{align*}
We use the linear edge element of the first family for $\boldsymbol{u}_h$, $\boldsymbol{v}_h$ and the linear Lagrange element for $q_h$, $\psi_h$. 
Denote by $\mathsf{u}$ and $\mathsf{q}$ the column vectors of the unknown coefficients of $\boldsymbol{u}_h$ and $q_h$, respectively,
and by $\mathsf{u_b}$ and $\mathsf{u_i}$ the parts of $\mathsf{u}$ that belong to the boundary and interior degrees of freedom, respectively. 
Then the matrix form of \eqref{E:eig_u_h} reads
\begin{align}
\left(
\begin{matrix}
\mathsf{A_{ii}}& \mathsf{A_{ib}}& \mathsf{O}
\\
\mathsf{A_{bi}}& \mathsf{A_{bb}}& \mathsf{O}
\\
\mathsf{O}& \mathsf{O}& \mathsf{O}
\end{matrix}
\right)
\left(
\begin{matrix}
\mathsf{u_i}
\\
\mathsf{u_b}
\\
\mathsf{q}
\end{matrix}
\right)
= -\lambda_h \left(
\begin{matrix}
\mathsf{O}& \mathsf{O}& \mathsf{O}
\\
\mathsf{O}& \mathsf{O}& \mathsf{B_b}
\\
\mathsf{O}& \mathsf{B_b^{\top}}& -\mathsf{M_{H_1}}
\end{matrix}
\right)
\left(
\begin{matrix}
\mathsf{u_i}
\\
\mathsf{u_b}
\\
\mathsf{q}
\end{matrix}
\right),
\label{E:matrix}
\end{align}
where the subscripts $\mathsf{_b}$ and/or $\mathsf{_i}$ stand for the interior or boundary indices. 
Note that $\mathsf{B_i}$ (which represents $\langle \textbf{curl}\, q_h,\boldsymbol{v}_{h,T}\rangle$ with interior bases $\boldsymbol{v}_h$) 
equals the zero matrix due to $\boldsymbol{v}_{h,T} = \boldsymbol{0}$ for interior bases $\boldsymbol{v}_h$ (see, e.g., Lemma 5.35 of \cite{Monk2003}). 
To use only the boundary degrees of freedom, we can write \eqref{E:matrix} as
\begin{align}
(\mathsf{A_{bb}} - \mathsf{A_{bi}}\mathsf{A_{ii}}^{\!\!-1}\mathsf{A_{ib}})\mathsf{u_b} = -\lambda_h \mathsf{B_b}\mathsf{M_{H_1}}^{\!\!\!-1}\mathsf{B_b^{\top}}\mathsf{u_b}.
\label{E:matrix2}
\end{align}
However, since \eqref{E:matrix2} contains the inverse of a matrix, we solve \eqref{E:matrix} rather than \eqref{E:matrix2}.

In the same way, the matrix form for the alternative eigenvalue problem \eqref{E:eig_u+_h} reads
\begin{align*}
\left(
\begin{matrix}
\mathsf{A_{ii}}& \mathsf{A_{ib}}& \mathsf{O}
\\
\mathsf{A_{bi}}& \mathsf{A_{bb}}& \mathsf{O}
\\
\mathsf{O}& \mathsf{O}& \mathsf{O}
\end{matrix}
\right)
\left(
\begin{matrix}
\mathsf{u_i}
\\
\mathsf{u_b}
\\
\mathsf{q}
\end{matrix}
\right)
= -\lambda^+_h \left(
\begin{matrix}
\mathsf{O}& \mathsf{O}& \mathsf{O}
\\
\mathsf{O}& \mathsf{M_{curl,bb}}& \mathsf{C_b}
\\
\mathsf{O}& \mathsf{C_b^{\top}}& \mathsf{M_{H_1}}
\end{matrix}
\right)
\left(
\begin{matrix}
\mathsf{u_i}
\\
\mathsf{u_b}
\\
\mathsf{q}
\end{matrix}
\right).
\end{align*}
And the corresponding compact form is
\begin{align*}
(\mathsf{A_{bb}} - \mathsf{A_{bi}}\mathsf{A_{ii}}^{\!\!-1}\mathsf{A_{ib}})\mathsf{u_b} = -\lambda^+_h (\mathsf{M_{curl,bb}} - \mathsf{C_b}\mathsf{M_{H_1}}^{\!\!\!-1}\mathsf{C_b^{\top}})\mathsf{u_b}.
\end{align*}

We show the average of the computed eigenvalues (and the convergence order) for the unit cube and the L-shaped domain
in Table \ref{t:cube} and Table \ref{t:L}, respectively. 
Since the exact eigenvalues of $\Omega_1$ and $\Omega_2$ are unknown, we use the relative errors and define the convergence order by
\begin{align*}
r_{\ell} = -\log\frac{|\lambda_{h_{\ell}}-\lambda_{h_{\ell-1}}|}{|\lambda_{h_{\ell-1}}-\lambda_{h_{\ell-2}}|}\left/\log\frac{\sqrt{N_{\ell}}}{\sqrt{N_{\ell-1}}}\right.,
\end{align*}
where $N_{\ell} \propto h_{\ell}^{-2}$ is the number of the edges on the boundary. 
\begin{table}[htp]
\centering
\begin{tabular}{ccc|c|c|c|c}
\hline
&&$N_{\ell}$& 360 & 1656 & 6174 & 23868
\\
\hline
\hline
Avg.\! of &$\lambda_{j,h_{\ell}}$ &$j = 1,2,3$&     -2.3373 &  -2.2184 &  -2.1840 (1.89)  & -2.1747 (1.92)
\\
\hline
&$\lambda^+_{j,h_{\ell}}$ & & -2.2288 &  -2.1862 &  -2.1747 (1.98) &  -2.1722 (2.32)
\\
\hline
\hline
&$\lambda_{j,h_{\ell}}$ &$j = 4,5$&     -3.0413 &  -2.6891 &  -2.6082 (2.24) &  -2.5875 (2.01)
\\
\hline
&$\lambda^+_{j,h_{\ell}}$ && -2.7418 &  -2.6199 &  -2.5893 (2.10)&  -2.5826 (2.25)
\\
\hline
\hline
&$\lambda_{j,h_{\ell}}$ &$j = 6,7,8$&   -6.5322 &  -5.5094 &  -5.1086 (1.42) &  -4.9932 (1.84)
\\
\hline
&$\lambda^+_{j,h_{\ell}}$ &&  -5.6449 &  -5.1702 &  -5.0162 (1.71)&  -4.9693 (1.76)
\\
\hline
\end{tabular}
\caption{Average of computed eigenvalues for $\Omega_{1}$ and the convergence orders.}
\label{t:cube}
\end{table}

\begin{table}[htp]
\centering
\begin{tabular}{ccc|c|c|c|c}
\hline
&&$N_{\ell}$& 405   &     1584   &     6183  &     24168
\\
\hline
\hline
Avg.\! of &$\lambda_{j,h_{\ell}}$ &$j = 1$&     -1.3769 &  -1.2488 &  -1.1799 (0.91) &  -1.1537 (1.41)
\\
\hline
&$\lambda^+_{j,h_{\ell}}$ & & -1.2714  & -1.2117 &  -1.1696 (0.52) &  -1.1505 (1.15)
\\
\hline
\hline
&$\lambda_{j,h_{\ell}}$ &$j = 2,3$&     -2.5634 &  -2.3926 &  -2.3381 (1.68)&  -2.3217 (1.76)
\\
\hline
&$\lambda^+_{j,h_{\ell}}$ && -2.3906  & -2.3420 &  -2.3237 (1.43)&  -2.3178 (1.67)
\\
\hline
\hline
&$\lambda_{j,h_{\ell}}$ &$j = 4$&-3.9772 &  -3.3877 &  -3.2077 (1.74)&  -3.1562 (1.83)
\\
\hline
&$\lambda^+_{j,h_{\ell}}$ &&  -3.2803 &  -3.2001 &  -3.1584 (0.96)&  -3.1420 (1.36)
\\
\hline
\end{tabular}
\caption{Average of computed eigenvalues for $\Omega_{2}$ and the convergence orders.}
\label{t:L}
\end{table}

For the unit cube, the convergence orders of the first two eigenvalues are approximately two, which is optimal.
While for the ``L-shaped'' case, because of the singularity of the domain, the convergence orders are deteriorated. 
Among all the eigenvalues the converging order of the first one is the lowest. 
This phenomenon is consistent with standard results for elliptic eigenvalue problems on reentrant domains.

For $\Omega_3$, since the exact eigenvalues $\lambda_*$'s are given in \cite{Camano2017}, the convergence order is defined as
\begin{align*}
r_{\ell} = -\log\frac{|\lambda_{h_{\ell}}-\lambda_*|}{|\lambda_{h_{\ell-1}}-\lambda_*|}\left/\log\frac{\sqrt{N_{\ell}}}{\sqrt{N_{\ell-1}}}\right..
\end{align*}

\begin{table}[htp]
\centering
\begin{tabular}{ccc|c|c|c|c}
\hline
&&$N_{\ell}$& 597    &    3276   &     9431   &    21185
\\
\hline
\hline
Avg.\! of &$\lambda_{j,h_{\ell}}$ &$j = 1,\dots,3$& -1.2034  &  -1.1185 (1.96) &  -1.1016 (2.05)&  -1.0956 (2.08)
\\
\hline
&$\lambda^+_{j,h_{\ell}}$ & & -1.1934 &  -1.1163 (1.94)&  -1.1007 (2.05)&  -1.0951 (2.08)
\\
\hline
\hline
&$\lambda_{j,h_{\ell}}$ &$j = 4,\dots,8$&  -2.7631 &  -2.4809 (2.04)&  -2.4293 (2.02)&  -2.4106 (2.09)
\\
\hline
&$\lambda^+_{j,h_{\ell}}$ && -2.5370 &  -2.4298 (1.85)&  -2.4073 (1.93)&  -2.3985 (2.12)
\\
\hline
\hline
&$\lambda_{j,h_{\ell}}$ &$j = 36,\dots,48$&  -24.1948 &  -9.5396 (2.67)&  -7.9114 (2.16)&  -7.3732 (2.09)
\\
\hline
&$\lambda^+_{j,h_{\ell}}$ &&  -9.8043 &  -7.8168 (1.52)&  -7.2672 (1.75)&  -7.0371 (1.92)
\\
\hline
\end{tabular}
\caption{Average of computed eigenvalues for $\Omega_{3}$ and the convergence orders.}
\label{t:ball}
\end{table}
Similar results are observed in Table \ref{t:ball}. 
The convergence orders are all approximately two, which is optimal.

\section{Conclusion and Future Work}
In this paper, we propose a finite element method for a Maxwell's equation with surface-divergence-free Neumann data. 
The discrete compactness property of the edge element spaces is proved and used to derive the error estimate. 
Furthermore, we show the convergence of a finite element method for the modified Maxwell's Stekloff eigenvalue problem. 

The convergence order we have proved is suboptimal, which is partially owing to the lack of sharp regularity results.
We plan to investigate the possibility to improve the order.
Another interesting problem is the error estimate for the finite element method using \eqref{E:eig_u+_h} proposed in \cite{Camano2017}. 
The numerical examples suggest that this method converges and possesses correct convergence order, which makes it worthwhile for further study. 

\paragraph{Acknowledgements}
The research of B. Gong is supported partially by China Postdoctoral Science Foundation Grant 2019M650460. The research of X. Wu is supported partially by the NSFC grants 11971120 and 91730302, and by Shanghai Science and Technology Commission Grant 17XD1400500.

\end{document}